\documentclass[11pt]{amsart}
\usepackage{latexsym,amscd,amssymb, graphicx, color, amsthm, bm, amsmath, cancel, enumitem, ytableau}  
\usepackage{tikz}
\usepackage[margin=1in]{geometry}
\usepackage{hyperref}
\usepackage[all,cmtip]{xy}
\usepackage{mathtools}
\usepackage{rotating}
\usetikzlibrary{math, arrows.meta}
\usetikzlibrary{calc}
\usepackage[backend=bibtex]{biblatex}
\numberwithin{equation}{section}

\newtheorem{theorem}{Theorem}[section]
\newtheorem{proposition}[theorem]{Proposition}
\newtheorem{corollary}[theorem]{Corollary}
\newtheorem{lemma}[theorem]{Lemma}
\newtheorem{conjecture}[theorem]{Conjecture}

\newtheorem{problem}[theorem]{Problem}
\newtheorem{example}[theorem]{Example}
\newtheorem{remark}[theorem]{Remark}
\newtheorem{definition}[theorem]{Definition}

\theoremstyle{definition}

\newcommand{\gl}{\mathrm{GL}}
\newcommand{\Frob}{{\mathrm{Frob}}}
\newcommand{\symm}{{\mathfrak{S}}}
\newcommand{\II}{{\mathbf{I}}}
\newcommand{\gr}{{\mathrm {gr}}}

\newcommand{\Res}{\mathrm{Res}}

\newcommand{\grFrob}{{\mathrm{grFrob}}}

\newcommand{\ZZZ}{{\mathcal{Z}}}

\newcommand{\MMM}{{\mathcal{M}}}
\newcommand{\PM}{{\mathcal{PM}}}

\newcommand{\CC}{{\mathbb{C}}}

\newcommand{\Mat}{{\mathrm{Mat}}}

\newcommand{\xxx}{{\mathbf{x}}}
\newcommand{\zzz}{{\mathbf{z}}}

\newcommand{\mmm}{{\mathfrak{m}}}

\newcommand{\CCgraph}[1]{\CC^{\binom{[#1]}{2}}}
\newcommand{\xxxgraph}[1]{\xxx_{\binom{[#1]}{2}}}

\newcommand{\ind}{\mathrm{Ind}}

\title{Plethysm and orbit harmonics}

\author{Hai Zhu}
\address{Department of Mathematics, UC San Diego, La Jolla, CA, 92093, USA}
\email{haz138@ucsd.edu}
\date{\today}

\begin{document}

\begin{abstract}
    Let $\Pi_{(b^a)}$ be the locus of unordered set partitions of $[ab]$ with $a$ blocks of size $b$. We embed unordered set partitions of $[n]$ into the affine space $\mathbb{C}^{\binom{[n]}{2}}$ with coordinate ring $\mathbb{C}\Big[\xxx_{\binom{[n]}{2}}\Big]$. Then, we apply orbit harmonics to $\Pi_{(2^a)}$ and $\Pi_{(a^2)}$, yielding graded $\symm_{2a}$-modules whose graded character formulae respectively refine the Schur expansions of $h_a[h_2]$ and $h_2[h_a]$ according to $\lambda_1$. We further extend this $\lambda_1$-separation phenomenon to quotients of $\mathbb{C}^{\binom{[n]}{2}}$ where $n$ is odd. Combining $\Pi_{(b^a)},\Pi_{(a^b)}$ and orbit harmonics, we propose a conjecture related to Foulkes' conjecture, and we prove the special case $b=2$. We also apply orbit harmonics to the locus $\Pi_{n,m}$ of unordered set partitions of $[n]$ without blocks of size greater than $m$, yielding a graded $\symm_n$-module $R(\Pi_{n,m})$. We determine the standard monomial basis of $R(\Pi_{n,m})$ with respect to any monomial order, as well as its graded character formula.
\end{abstract}

\maketitle

\section{Introduction}\label{sec:intro}
Let $\xxx_N=(x_1,\dots,x_N)$ be a sequence of variables and let $\CC[\xxx_N]$ be the polynomial ring over these variables. Given a finite locus $\ZZZ\subseteq\CC^N$, let $\II(\ZZZ)\subseteq\CC[\xxx_N]$ be its vanishing ideal given by
\[\II(\ZZZ)\coloneqq\{f\in\CC[\xxx_N]\,:\,\text{$f(\zzz)=0$ for all $\zzz\in\ZZZ$}\}.\]
The orbit harmonics method associates to $\II(\ZZZ)$ a homogeneous ideal $\gr\II(\ZZZ)\subseteq\CC[\xxx_N]$ and produces a graded quotient ring $R(\ZZZ)\coloneqq\CC[\xxx_N]/\gr\II(\ZZZ)$. There exists a chain of vector space identifications
\begin{align}\label{eq:intro-chain}
    \CC[\ZZZ]\cong\CC[\xxx_N]/\II(\ZZZ)\cong R(\ZZZ).
\end{align}
If $\ZZZ$ is stable under the action of a matrix group $G\subseteq\gl_N(\CC)$, then the chain \eqref{eq:intro-chain} consists of $G$-module isomorphisms. While $\CC[\ZZZ]$ is an ungraded $\symm_n$-module, $R(\ZZZ)$ is a graded $\symm_n$-module that refines $\CC[\ZZZ]$. This graded structure is governed by the combinatorics of $\ZZZ$.

Geometrically, the orbit harmonics method is a linear deformation from a finite locus $\ZZZ\subseteq\CC^N$ to a scheme supported at the origin with multiplicity $\lvert\ZZZ\rvert$. For example, let $\symm_3$ act on $\CC^3$ by permuting components. The orbit harmonics deformation of a generic $\symm_3$-orbit $\ZZZ$ is shown below. ($\ZZZ$ is contained in a hyperplane, and the $\symm_3$-action is generated by reflections in the lines below).
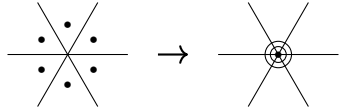
\begin{center}
 \begin{tikzpicture}[scale = 0.2]
\draw (-4,0) -- (4,0);
\draw (-2,-3.46) -- (2,3.46);
\draw (-2,3.46) -- (2,-3.46);

 \fontsize{5pt}{5pt} \selectfont
\node at (0,2) {$\bullet$};
\node at (0,-2) {$\bullet$};

\node at (-1.73,1) {$\bullet$};
\node at (-1.73,-1) {$\bullet$};
\node at (1.73,-1) {$\bullet$};
\node at (1.73,1) {$\bullet$};

\draw[thick, ->] (6,0) -- (8,0);

\draw (10,0) -- (18,0);
\draw (12,-3.46) -- (16,3.46);
\draw (12,3.46) -- (16,-3.46);

\draw (14,0) circle (15pt);
\draw(14,0) circle (25pt);
\node at (14,0) {$\bullet$};

 \end{tikzpicture}
\end{center}

The application of orbit harmonics occurs in various topics, such as presentations of cohomology rings \cite{MR1168926}, Macdonald theory \cite{MR4223028,MR3783430}, cyclic sieving \cite{MR4359538}, Donaldson–Thomas theory \cite{MR4675069}, and Ehrhart theory \cite{reiner2024harmonicsgradedehrharttheory}.

Rhoades~\cite{MR4821538} initiated the application of the orbit harmonics method to matrix loci. He considered the permutation matrix locus $\symm_n\subseteq\Mat_{n\times n}(\CC)$, showing that the graded structure of $R(\symm_n)$ is governed by the Viennot shadow line avatar of the Schensted correspondence and the first-row length $\lambda_1$ arising from Specht modules $V^\lambda$. His work was extended to other matrix loci in \cite{MR4936857,liu2025extensionviennotsshadowrook,MR4990066,zhu2025rookplacementsorbitharmonics}. Interestingly, Liu, Ma, Rhoades, Zhu~\cite{MR4887799} found that the ``$\lambda_1$-governs-the-graded-structure'' property for $R(\symm_n)$, observed by Rhoades, also holds for the perfect matching locus $\PM_{n}\subseteq\Mat_{n\times n}(\CC)$. Specifically, for an even integer $n$, $\PM_n$ consists of all the symmetric permutation matrices corresponding to fixed-point-free involutions, and the graded $\symm_n$-module character of $R(\PM_n)$ refines the Schur expansion of the plethysm $h_{n/2}[h_2]$ according to $\lambda_1$ arising from the Schur function $s_\lambda$. We extend this orbit harmonics construction to the plethysm $h_2[h_{n/2}]$, yielding a graded $\symm_n$-module refining the Schur expansion of $h_2[h_{n/2}]$ according to $\lambda_1$ as well. Furthermore, we extend our ring structure to odd integers $n$, obtaining two families of graded $\symm_n$-modules whose graded characters are also governed by $\lambda_1$. 

Unlike general matrix loci, we only need the strictly upper triangular part. Consequently, we focus on the affine space $\CC^{\binom{[n]}{2}}$ with coordinate ring $\CC\Big[\xxxgraph{n}\Big]$. We embed all unordered set partitions of $[n]$ into $\CC^{\binom{[n]}{2}}$. For $\mu\vdash n$, let $\Pi_\mu$ be the locus of unordered set partitions of $[n]$ whose block sizes fit in with $\mu$. Let $\Pi_{n,m}$ be the locus of unordered set partitions of $[n]$ whose block sizes are not greater than $m$. Both loci are stable under the $\symm_n$-action, so $R(\Pi_\mu)$ and $R(\Pi_{n,m})$ are graded $\symm_n$-modules. In particular, $\Frob(R(\Pi_{(b^a)}))=h_a[h_b]$, which gives a graded refinement of the plethysm $h_a[h_b]$. Therefore, orbit harmonics may help us understand plethysm and Foulkes' conjecture. We raise Conjecture~\ref{conj:contain} which is stronger than Foulkes' conjecture, and we can solve a special case of our conjecture (see Theorem~\ref{thm:contain_b=2}). Our main results are stated below.

We find three graded character formulae (see Theorems~\ref{thm:grfrob_h_a[h_2]}, \ref{thm:grad_frob_h_2[h_a]} and Proposition~\ref{prop:general-module-str})
\begin{align}&\label{eq:intro-h_a[h_2]}
    \grFrob(R(\Pi_{(2^a)});q)=\sum_{\substack{\lambda\vdash 2a\\ \text{$\lambda$ is even}}} q^{\frac{2a-\lambda_1}{2}}\cdot s_\lambda\\
&\label{eq:intro-h_2[h_a]}
    \grFrob(R(\Pi_{(a^2)});q) = \sum_{d=0}^{\lfloor a/2\rfloor}q^d\cdot s_{(2a-2d,2d)}
\\
&\label{eq:intro-Pi_{n,m}}
    \Frob(R(\Pi_{n,m})_d)=\sum_{\substack{\lambda\vdash n\\ \lambda_1\le m \\ \ell(\lambda)=n-d}}\prod_{i\ge 1}h_{m_i(\lambda)}[h_i]
\end{align}
where $m_i(\lambda)$ is the multiplicity of $i$ in $\lambda$. Note that Equation~\eqref{eq:intro-h_a[h_2]} revisits \cite[Theorem 4.4]{MR4887799}. Moreover, we respectively extend the quotient ring structures of $R(\Pi_{(2^a)})=\CC\Big[\xxxgraph{2a}\Big]\Big/\gr\II(\Pi_{(2^a)})$ and $R(\Pi_{(a^2)})=\CC\Big[\xxxgraph{2a}\Big]\Big/\gr\II(\Pi_{(a^2)})$ to obtain $\CC\Big[\xxxgraph{n}\Big]\Big/I(n)$ and $\CC\Big[\xxxgraph{n}\Big]\Big/J(n)$ where $n$ may be odd. The graded character formulae of these two new quotients respectively generalize Equations~\eqref{eq:intro-h_a[h_2]} and \eqref{eq:intro-h_2[h_a]}, which are also governed by $\lambda_1$ (see Theorems~\ref{thm:generalize-h_a[h_2]} and \ref{thm:generalize-h_2[h_a]}):
\begin{align}
    &\label{eq:intro-general-h_a[h_2]}\grFrob(\CC\Big[\xxxgraph{n}\Big]\Big/I(n);q) = \sum_{\substack{\lambda\vdash n \\ \text{$2\mid\lambda_i$ for $i>1$}}} q^{\frac{n-\lambda_1}{2}} \cdot s_\lambda \\
    &\label{eq:intro-general-h_2[h_a]}\grFrob\Big(\CC\Big[\xxxgraph{n}\Big]\Big/J(n);q\Big) = \sum_{d=0}^{\lfloor n/4 \rfloor} q^d \cdot s_{(n-2d,2d)}.
\end{align}
Note that \eqref{eq:intro-general-h_2[h_a]} indicates that $\CC\Big[\xxxgraph{n}\Big]_d\Big/J(n)_d\cong V^{(n-2d,2d)}$, which provides a new model to understand the Specht module $V^{(n-2d,2d)}$.

For Gröbner theory and ring structure, we find the standard monomial bases of $R(\Pi_{n,m})$ with respect to any monomial order (see Theorem~\ref{thm:general-satandard}). We also give explicit presentations (i.e. concise generating set of defining ideals) of the orbit harmonics rings $R(\Pi_{(2^a)}),R(\Pi_{(a^2)}),R(\Pi_{n,m})$ (see Theorem~\ref{thm:grfrob_h_a[h_2]}, Proposition~\ref{prop:ideal_h_2[h_a]}, and Corollary~\ref{cor:basis-presentation}).

The rest of this paper is structured as follows. In Section~\ref{sec:background}, we introduce some necessary background knowledge about Gröbner theory, orbit harmonics, symmetric functions and symmetric group representation theory. In Section~\ref{sec:main}, we first describe how to embed unordered set partitions into $\CC^{\binom{[n]}{2}}$ and relevant $\symm_n$-actions, then study $R(\Pi_{(2^a)})$ and $R(\Pi_{(a^2)})$ in Subsection~\ref{subsec:b=2}, extend their ring structures and graded module structures in Subsection~\ref{subsec:length-generalize}, and study $R(\Pi_{n,m})$ in Subsection~\ref{subsec:general}. In Section~\ref{sec:conclusion}, we present future directions and open problems.

\section{Background}\label{sec:background}

\subsection{Gröbner Theory}\label{subsec:grobner}
Let $\xxx_N = (x_1, x_2, \dots x_N)$ be a sequence of $N$ variables, and let $\CC[\xxx_N]$ be the polynomial ring over these variables. A total order $<$ on the set of monomials in $\CC[\xxx_N]$ is called a {\em monomial order} if
\begin{itemize}
    \item $1 \leq m$ for every monomial $m \in \CC[\xxx_N]$
    \item for monomials $m, \, m_1, \, m_2$, we have that $m_1 < m_2$ implies $mm_1 < mm_2$.
\end{itemize}

Let $<$ be a monomial order. For any nonzero polynomial $f \in \CC[\xxx_N]$, the {\em initial monomial} $\text{in}_{<} f$ of $f$ is the largest monomial in $f$ with respect to $<$. Given an ideal $I \subseteq \CC[\xxx_N]$, define its {\em initial ideal} by
\begin{equation*}
    \text{in}_{<} I = \langle \text{in}_{<}f : f \in I, f\neq 0 \rangle.
\end{equation*}
A monomial $m \in \CC[\xxx_N]$ is called a {\em standard monomial} of $I$ if it is not an element of $\text{in}_{<} I$. It is well known that
the collection of all the standard monomials of $I$ descends to a basis of the vector space $\CC[\xxx_N]/I$. This is called the {\em standard monomial basis}. See \cite{MR4952933} for further details on Gröbner theory.

\subsection{Orbit harmonics}\label{subsec:orbit-harmonics}
Let $\ZZZ\subseteq\CC^N$ be a finite locus. Let $\CC[\xxx_N]$ be the coordinate ring of the affine space $\CC^N$ where $\xxx_N=(x_1,\dots,x_N)$ is the sequence of all the variables. The \emph{vanishing ideal} $\II(\ZZZ)\subseteq\CC[\xxx_N]$ of $\ZZZ$ is given by
\[\II(\ZZZ)\coloneqq\{f\in\CC[\xxx_N]\,:\,\text{$f(\zzz)=0$ for all $\zzz\in\ZZZ$}\}.\]
Let $\CC[\ZZZ]$ be the vector space consisting of all the functions $f\,:\,\ZZZ\rightarrow\CC$. Multivariate Lagrange interpolation gives the vector space isomorphism
\begin{align}\label{eq:ungrad-isom}
    \CC[\ZZZ]\cong\CC[\xxx_N]/\II(\ZZZ).
\end{align}

Given any nonzero polynomial $f\in\CC[\xxx_N]$, we write $f=f_0+\dots+f_d$ where $f_k$ is homogeneous for $1\le k\le d$ and $f_d\neq 0$. Denote the top-degree homogeneous part $f_d$ of $f$ by $\tau(f)\coloneqq f_d$. For any ideal $I\subseteq\CC[\xxx_N]$, the \emph{associated graded ideal} $\gr I$ of $I$ is the homogeneous ideal
\[\gr I\coloneqq \langle \tau(f)\,:\,0\neq f\in I\rangle.\]
\begin{remark}\label{rmk:generator-difficult}
    Given a generating set $I=\langle g_1,\dots,g_m\rangle$, the containment $\langle \tau(g_1),\dots,\tau(g_m)\rangle\subseteq\gr I$ is usually strict. However, the equality holds if $g_1,\dots,g_m$ form a Gröbner basis of $I$. In general, finding a concise generating set for $\gr I$ is difficult. 
\end{remark}

The orbit harmonics method yields a graded ring $R(\ZZZ)\coloneqq\CC[\xxx_N]/\gr\II(\ZZZ)$. A standard result for orbit harmonics extends the vector space isomorphism \eqref{eq:ungrad-isom} to a chain of vector space isomorphisms
\begin{align}\label{eq:grad-isom}
    \CC[\ZZZ]\cong\CC[\xxx_N]/\II(\ZZZ)\cong R(\ZZZ).
\end{align}
If $\ZZZ$ is stable under the action of a subgroup $G\subseteq \gl(\CC^N)$, we can further regard \eqref{eq:grad-isom} as a chain of $G$-module isomorphisms. Interestingly, $R(\ZZZ)$ has the additional structure of a graded $G$-module, which provides a graded refinement for $\CC[\ZZZ]$.

\subsection{Combinatorics}\label{subsec:combinatorics}

For any nonnegative integer $n$, a \emph{partition} of $n$ is a weakly decreasing sequence $\lambda$ of nonnegative integers
\[\lambda = (\lambda_1\ge\lambda_2\ge\dots\ge\lambda_m)\]
with $\sum_{i=1}^m \lambda_i = n$. Here, we call $n$ the \emph{size} of $\lambda$, which is denoted by $\lvert\lambda\rvert = n$ or $\lambda\vdash n$. One abuse of notation is that we allow appending several zeros to $\lambda$, yielding the identification
\[(\lambda_1,\dots,\lambda_m) = (\lambda_1,\dots,\lambda_m,0,\dots,0).\]
The components $\lambda_i$ are called the \emph{parts} of $\lambda$. In particular, we say that $\lambda$ is \emph{even} if all of its parts are even.
\begin{remark}\label{rmk:repeat}
We may write a repeated-part subsequence $\overbrace{b,b,\dots,b}^a$ for short by $b^a$, such as the abbreviation $(5,5,5,2,2,2,1) = (5^3,2^3,1)$. The reader should be careful not to mix this notation up with power. We never denote power by $b^a$ in this paper.
\end{remark}
Given a partition $\lambda=(\lambda_1,\dots,\lambda_m)$, its \emph{length} $\ell(\lambda)$ is the maximal index $k$ such that $\lambda_k\neq 0$. We always identify a partition $\lambda$ with its \emph{Young diagram}, a diagram with $\lambda_i$ left-justified cells on the $i$-th row for all $i\ge 1$. Taking the transpose of the Young diagram of $\lambda$, we obtain the \emph{conjugate} $\lambda^\prime$ of $\lambda$. For example, the Young diagrams of $\lambda=(4,3,1)\vdash 8$ and its conjugate $\lambda^\prime\vdash 8$ are shown below. In this example, we have $\ell(\lambda)=3$ and $\ell(\lambda^\prime)=\lambda_1=4$.
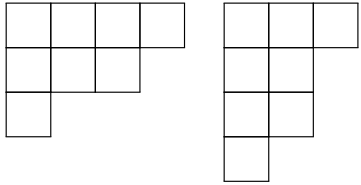
\begin{center}
    \begin{ytableau}
    \: & \: & \: & \: \cr
    \: & \: & \: \cr
    \: \cr
\end{ytableau}
\quad
    \begin{ytableau}
    \: & \: & \: \cr
    \: & \: \cr
    \: & \: \cr
    \:
\end{ytableau}
\end{center}

A \emph{skew partition} $\lambda/\mu$ is a pair of partitions $\lambda$ and $\mu$ such that $\mu\subseteq\lambda$. We also identify $\lambda/\mu$ with its \emph{Young diagram} given by removing the cells of $\mu$ from $\lambda$. The \emph{size} $\lvert\lambda/\mu\rvert$ of $\lambda/\mu$ is given by $\lvert\lambda/\mu\rvert = \lvert\lambda\rvert-\lvert\mu\rvert$. In particular, $\lambda/\mu$ is called a \emph{horizontal strip} if it has at most one cell in every column. For example, we consider the inner partition $\mu=(3,3,1,1)$ and the outer partition $\lambda=(4,3,3,1)$, then the Young diagram $\lambda/\mu$ is marked with bullets $\bullet$ shown below.
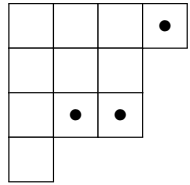
\[\ydiagram[*(white) \bullet]
    {3+1,3+0,1+2,1+0}
    *[*(white)]{3,3,1,1}\]
As $\lambda/\mu$ has at most one cell in each column, $\lambda/\mu$ is a horizontal strip.

\subsection{Representation theory of symmetric groups}\label{subsec:rep}
Let $\Lambda = \bigoplus_{n\ge 0} \Lambda_n$ be the graded ring of symmetric functions in infinitely many variables $x_1,x_2,\dots$ over the ground field $\CC(q)$. The space $\Lambda_n$ of symmetric functions of degree $n$ admits several well-known $\CC(q)$-linear bases indexed by partitions of $n$: the \emph{monomial basis} $\{m_\lambda\}_{\lambda\vdash n}$, the \emph{complete homogeneous basis} $\{h_\lambda\}_{\lambda\vdash n}$, the \emph{elementary basis} $\{e_\lambda\}_{\lambda\vdash n}$, the \emph{power sum basis} $\{p_\lambda\}_{\lambda\vdash n}$, and the \emph{Schur basis} $\{s_\lambda\}_{\lambda\vdash n}$. See \cite{MR3443860} for their definitions and properties. We will primarily use the Schur basis, the complete homogeneous basis, and the power sum basis, and we write $h_d,p_d$ respectively for $h_{(d)},p_{(d)}$.

A symmetric function $F\in\Lambda$ is \emph{Schur-positive} if its Schur expansion $F=\sum_{\lambda}c_\lambda (q)\cdot s_\lambda$ has nonnegative coefficients $c_\lambda (q) \in \mathbb{R}_{\ge 0}[q]$. We assign a partial order $\le$ to $\Lambda$ by
\[\text{$F\le G$ if and only if $G-F$ is Schur-positive}.\]
We also define the \emph{truncation operator} $\{-\}_P$ where $P$ is a condition. For $F = \sum_\lambda c_\lambda(q)\cdot s_\lambda\in\Lambda$, we define $\{F\}_P$ by
\[\{F\}_P\coloneqq\sum_{\text{$\lambda$ satisfies $P$}}c_\lambda(q)\cdot s_\lambda.\]
We will use the truncation operator to restrict the first row lengths of partitions, e.g.
\[\{F\}_{\lambda_1 \le a} = \sum_{\lambda_1 \le a} c_\lambda(q)\cdot s_\lambda.\]

The multiplication of Schur functions by complete homogeneous functions satisfies the Pieri rule: For any partition $\mu$ and any integer $a\ge 0$, we have that
\[s_\mu\cdot h_a = \sum_{\lambda/\mu}s_\lambda\] summing over horizontal strips $\lambda/\mu$ of size $a$.

In addition to the ordinary multiplication, the symmetric function ring $\Lambda$ carries a binary operation called \emph{plethysm}, which is denoted by $(F,G)\mapsto F[G]$. For $G\in\Lambda$, we have $p_n[G] = G(x_1^n,x_2^n,\dots)$. As $\{p_1,p_2,\dots\}$ form an algebraically independent generating set of $\Lambda$, we can compute $F[G]$ for all $F,G\in\Lambda$ using the following properties
\begin{align*}
    \begin{cases}
        (F_1 + F_2 )[G] = F_1[G] + F_2[G] \\
        (F_1\cdot F_2)[G] = F_1[G]\cdot F_2[G] \\
        c[G]= c
    \end{cases}
\end{align*}
for all $F_1,F_2,G\in\Lambda$ and $c\in\CC(q)$. Please refer to \cite{MR3443860} for more details.

Consider the symmetric group $\symm_n$. All the irreducible $\symm_n$-modules are in one-to-one correspondence with all the partitions of $n$, called \emph{Specht modules}. The Specht module corresponding to $\lambda\vdash n$ is denoted by $V^\lambda$. Given an $\symm_n$-module $V\cong\bigoplus_{\lambda\vdash n}(V^\lambda)^{\oplus c_\lambda}$ with unique multiplicities $c_\lambda\in\mathbb{Z}_{\ge 0}$, we study the $\symm_n$-module structure of $V$ using its \emph{Frobenius image} $\Frob(V)\in\Lambda$ given by
\[\Frob(V)\coloneqq\sum_{\lambda\vdash n}c_{\lambda}\cdot s_\lambda.\]
If $V=\bigoplus_{d=0}^m V_d$ is a graded $\symm_n$-module, we define its \emph{graded Frobenius image} $\grFrob(V;q)\in\Lambda$ by
\[\grFrob(V;q)\coloneqq\sum_{d=0}^m q^d\cdot \Frob(V_d).\]

For integers $n,m\ge 0$, we embed $\symm_n\times\symm_m$ into $\symm_{n+m}$ by letting $\symm_n$ act on $\{1,\dots,n\}$ and letting $\symm_m$ act on $\{n+1,\dots,n+m\}$. Therefore, if $V$ is an $\symm_n$-module and $W$ is an $\symm_m$-module, then $V\otimes W$ is naturally an $\symm_n\times\symm_m$-module so that we can regard the induction representation $\ind_{\symm_n\times\symm_m}^{\symm_{n+m}}(V\otimes W)$. The \emph{induction product} of $V$ and $W$ is defined by
\[V\circ W\coloneqq\ind_{\symm_n\times\symm_m}^{\symm_{n+m}}(V\otimes W).\]
Induction product is compatible with multiplication of symmetric functions (see e.g. \cite{MR3443860,sagan2013symmetric}):
\[\Frob(V\circ W) = \Frob(V)\cdot\Frob(W).\]

For integers $n,m\ge 0$, the \emph{wreath product} $\symm_n\wr\symm_m$ is the subgroup of $\symm_{nm}$ generated by:
\begin{itemize}
    \item the $n$-fold direct product $(\symm_m)^n$ where the $i$-th copy of $\symm_m$ acts on $\{(i-1)m+1,(i-1)m+2,\dots,im\}$ for all $1\le i\le n$, and
    \item all the products of transpositions \[((i-1)m+1,(j-1)m+1)((i-1)m+2,(j-1)m+2)\cdots(im,jm)\]
    for $1\le i,j\le n$.
\end{itemize}
The wreath product $\symm_n\wr\symm_m$ may be intuitively thought of as follows. Fill the cells of an $n\times m$ grid with $1,2,\dots,nm$ in English reading order (see the figure below for an example with $n=3,m=4$). Then $\symm_n\wr\symm_m$ is the subgroup of $\symm_{nm}$ generated by permutations which either permute entries in one row or permute rows without rearranging entries within rows.
\begin{center}
    \begin{ytableau}
    1 & 2 & 3 & 4 \cr
    5 & 6 & 7 & 8 \cr
    9 & 10 & 11 & 12
\end{ytableau}
\end{center}
Let $V$ be an $\symm_n$-module and let $W$ be an $\symm_m$-module. The $(\symm_n\wr\symm_m)$-module $V\wr W$ is the vector space $V\otimes (W^{\otimes n})$ carrying the $(\symm_n\wr\symm_m)$-action by
\[\begin{cases}
    (u_1,\dots,u_n)\cdot (v\otimes w_1\otimes\dots\otimes w_n) \coloneqq v\otimes (u_1\cdot w_1)\otimes\dots\otimes (u_n\cdot w_n), & (u_1,\dots,u_n)\in (\symm_m)^n \\
    g\cdot (v\otimes w_1\otimes\dots\otimes w_n) \coloneqq (g\cdot v)\otimes w_{g^{-1}(1)}\otimes\dots\otimes w_{g^{-1}(n)}, & g\in\symm_n.
\end{cases}\]
Wreath product aligns with plethysm of symmetric functions (see \cite{MR3443860}):
\[\Frob(\ind_{\symm_n\wr\symm_m}^{\symm_{nm}}(V\wr W)) = \Frob(V)[\Frob(W)].\]

For any partition $\lambda\vdash n$, let $\Pi_\lambda$ be the set of unordered set partitions of $[n]$ with block sizes corresponding to $\lambda$. That is, we have
\[\Pi_\lambda \coloneqq \bigg\{\{B_1,\dots,B_{\ell(\lambda)}\}\,:\,\text{$\bigsqcup_{i=1}^{\ell(\lambda)}B_i = [n]$ where $\lvert B_i \rvert = \lambda_i$ for $1\le i\le \ell(\lambda)$}\bigg\}.\]
Plethysm can be used to study the $\symm_n$-representation theory with respect to $\Pi_\lambda$.
\begin{lemma}\label{lem:plethysm-set-partition}
    Let $\symm_{nm}$ act on $\Pi_{(m^n)}$ by permuting the elements of $[nm]$, yielding an $\symm_{nm}$-module structure of $\CC[\Pi_{(m^n)}]$. Then we have
    \[\Frob(\CC[\Pi_{(m^n)}]) = h_n[h_m].\]
\end{lemma}
\begin{proof}
    Since the stabilizer of $\Pi_{(m^n)}$ is $\symm_n\wr\symm_m\subseteq\symm_{nm}$, we have an $\symm_{nm}$-equivariant bijection
    \[\Pi_{(m^n)}\overset{1:1}{\rightarrow}\symm_{nm}/(\symm_n\wr\symm_m).\]
    Then we have a chain of $\symm_{nm}$-module isomorphisms
    \begin{align*}
        \CC[\Pi_{(m^n)}] \cong \CC[\symm_{nm}/(\symm_n\wr\symm_m)] \cong\ind_{\symm_n\wr\symm_m}^{\symm_{nm}}\mathbf{1}_{\symm_n\wr\symm_m}\cong\ind_{\symm_n\wr\symm_m}^{\symm_{nm}}(\mathbf{1}_{\symm_n}\wr\mathbf{1}_{\symm_m}).
    \end{align*}
    It follows that
    \[\Frob(\CC[\Pi_{(m^n)}]) = \Frob(\ind_{\symm_n\wr\symm_m}^{\symm_{nm}}(\mathbf{1}_{\symm_n}\wr\mathbf{1}_{\symm_m})) = \Frob(\mathbf{1}_{\symm_n})[\Frob(\mathbf{1}_{\symm_m})] = h_n[h_m].\]
\end{proof}
Finding a positive formula for the Schur expansion of $h_a[h_b]$ is a dramatically difficult open problem. However, for the special case $a=2$ or $b=2$, $h_a[h_b]$ has concise Schur expansion:
\begin{align}\label{eq:h_a[h_2]}
    h_a[h_2] &= \sum_{\substack{\lambda\vdash 2a \\ \text{$\lambda$ is even}}} s_{\lambda} \\
    \label{eq:h_2[h_b]}
    h_2[h_b] &=\sum_{d=0}^{\lfloor b/2\rfloor} s_{(2b-2d,2d)}.
\end{align}

We will need one standard result to govern the length of the first row of a partition. For $1\le j\le n$, we use the embedding $\symm_j\subseteq\symm_n$ by letting $\symm_j$ permute $1,2,\dots,j\in[n]$. Let $\eta_j\in\CC[\symm_j]\subseteq\CC[\symm_n]$ be the symmetrizer
\[\eta_j\coloneqq\sum_{w\in\symm_j}w.\]
\begin{lemma}\label{lem:row-length}
    Let $1\le j\le n$. For a partition $\lambda\vdash n$, we have $\eta_j\cdot V^\lambda \neq 0$ if and only if $\lambda_1\ge j$.
\end{lemma}
\begin{proof}
    Note that $\eta_j\cdot V^\lambda = (V^\lambda)^{\symm_j} =\{v\in V^\lambda\,:\,\text{$w\cdot v = v$ for all $w\in\symm_j$}\}$. Therefore, we have that
    \[\text{$\eta_j\cdot V^\lambda \neq 0$ if and only if $(\Res_{\symm_j}^{\symm_n}(V^\lambda))^{\symm_j} \neq 0$}.\]
    Thus, $\eta_j\cdot V^\lambda \neq 0$ if and only if the multiplicity of $\mathbf{1}_{\symm_j}$ in $\Res_{\symm_j}^{\symm_n}(V^\lambda)$ is positive. The Branching Rule (see e.g. \cite{MR3443860,sagan2013symmetric}) indicates that
    \[\Res_{\symm_{n-1}}^{\symm_n}(V^\lambda) = \bigoplus_\mu V^\mu\]
    summing over all the partitions $\mu\vdash n-1$ obtained by removing one cell from $\lambda$. Iterating this rule, we deduce that $\mathbf{1}_{\symm_j} = V^{(j)}$ occurs in $\Res_{\symm_j}^{\symm_n}(V^\lambda)$ with positive multiplicity if and only if $\lambda_1\ge j$.
\end{proof}


\section{Plethysm and orbit harmonics}\label{sec:main}
Given an unordered set partition $\pi = \{B_1,\dots,B_\ell\}$ of $[n]$ (which means that $B_i\neq \varnothing$ for $1\le i\le\ell$ and $\bigsqcup_{i=1}^\ell B_i = [n]$), we identify $\pi$ with the point $\zzz=(z_{\{i,j\}})_{1\le i<j\le n}\in\CCgraph{n}$ given by
\[z_{\{i,j\}} = \begin{cases}
    1, &\text{if $\{i,j\}\subseteq B_k$ for some $1\le k\le\ell$} \\
    0, &\text{otherwise}.
\end{cases}.\]
Intuitively, we identify $\pi$ with a graph obtained by connecting all pairs of vertices in the same block of $\pi$, and then let $\zzz$ be half of the adjacency matrix of this graph which is the disjoint union of $\ell$ complete subgraphs. As a result, we can embed each locus $\ZZZ$ consisting of some unordered set partitions of $[n]$ into $\CCgraph{n}$ and apply orbit harmonics to it.

Let $\xxxgraph{n}$ be the array of variables $(x_{\{i,j\}})_{1\le i<j\le n}$ indexed by $\binom{[n]}{2}$. Let $\CC\Big[\xxxgraph{n}\Big]$ be the polynomial ring over these variables. If further $\ZZZ$ is stable under the $\symm_n$-action permuting $[n]$ (e.g. $\ZZZ=\Pi_\lambda$), we have an $\symm_n$-module identification $\CC[\ZZZ]\cong R(\ZZZ)$ (see \eqref{eq:grad-isom}) where the $\symm_n$-action on $R(\ZZZ)=\CC\Big[\xxxgraph{n}\Big]\Big/\gr\II(\ZZZ)$ is algebraically generated by $w\cdot x_{\{i,j\}} = x_{w(i),w(j)}$.
\begin{remark}\label{rmk:interchange-index}
    In our setting, the indices are unordered sets with $2$ elements. Therefore, we have that $x_{\{i,j\}} = x_{\{j,i\}}$.
\end{remark}

As mentioned in Lemma~\ref{lem:plethysm-set-partition}, plethysm is tightly related to unordered set partitions. It follows that
\[\Frob(R(\Pi_{(b^a)})) = h_a[h_b].\]
Therefore, the orbit harmonics method offers a potential approach to studying Foulkes' conjecture. Specifically, we give the following stronger conjecture.
\begin{conjecture}\label{conj:contain}
    For positive integers $a\ge b$, we have that
    \[\gr\II(\Pi_{(b^a)})\subseteq\gr\II(\Pi_{(a^b)}).\]
\end{conjecture}
In fact, once Conjecture~\ref{conj:contain} holds, then we would obtain an $\symm_n$-module surjection $R(\Pi_{(b^a)})\twoheadrightarrow R(\Pi_{(a^b)})$ and hence $h_a[h_b] = \Frob(R(\Pi_{(b^a)})) \ge \Frob(R(\Pi_{(a^b)})) =h_b[h_a]$. Currently, we can only solve the special case $b=2$ of Conjecture~\ref{conj:contain}. See Subsection~\ref{subsec:b=2} for details.

\subsection{Understanding $h_a[h_2]$ and $h_2[h_a]$ via orbit harmonics}\label{subsec:b=2}
For $h_a[h_2]$, the following result follows from \cite{MR4887799} and gives a graded refinement of $h_a[h_2]$ and a generating set of $\gr\II(\Pi_{(2^a)})$.
\begin{theorem}\label{thm:grfrob_h_a[h_2]}
    For all positive integers $a$, we have that
    \[\grFrob(R(\Pi_{(2^a)});q) = \sum_{\substack{\lambda\vdash 2a \\ \text{$\lambda$ is even}}} q^{\frac{2a-\lambda_1}{2}} \cdot s_{\lambda}\]
    and that the defining ideal $\gr\II(\Pi_{(2^a)})$ of $R(\Pi_{(2^a)})$ is generated by
    \begin{itemize}
        \item all products $x_S\cdot x_T$ for $S,T\in\binom{[2a]}{2}$ with $S\cap T\neq\varnothing$ ($S$ may equal $T$), and
        \item all sums $\sum_{\substack{j\in[2a] \\ j\neq i}} x_{\{i,j\}}$ for $i\in[2a]$.
    \end{itemize}
\end{theorem}
\begin{proof}
    Let $I_a\subseteq\CC\Big[\xxxgraph{2a}\Big]$ be the ideal generated by the two families of generators above. We first show that
    \begin{align}\label{eq:contain_h_a[h_2]}
        I_a\subseteq\gr\II(\Pi_{(2^a)}).
    \end{align}
    For the first type of generators $x_S\cdot x_T$ with $S\cap T\neq \varnothing$, if $S\neq T$ then we have that $x_S\cdot x_T\in\II(\Pi_{(2^a)})$ because any two distinct blocks of the same unordered set partition are disjoint. If $S=T$, we have that $x_S\cdot x_T-x_S = x_S\cdot (x_S -1)\in\II(\Pi_{(2^a)})$ because all the components of any $\zzz\in\Pi_{(2^a)}$ take values in $\{0,1\}$. Regardless of whether $S$ equals $T$, we have that $x_S\cdot x_T\in\gr\II(\Pi_{(2^a)})$. For the second type of generators $\sum_{\substack{j\in[2a] \\ j\neq i}} x_{\{i,j\}}$, we deduce that $\sum_{\substack{j\in[2a] \\ j\neq i}} x_{\{i,j\}} -1 \in \II(\Pi_{(2^a)})$ from the fact that for any $\pi\in\Pi_{(2^a)}$ each $i\in[2a]$ belongs to exactly one block of $\pi$. It follows that $\sum_{\substack{j\in[2a] \\ j\neq i}} x_{\{i,j\}}\in\gr\II(\Pi_{(2^a)})$. Thus, the containment \eqref{eq:contain_h_a[h_2]} is true.

    An immediate corollary of \eqref{eq:contain_h_a[h_2]} is the natural graded $\symm_{2a}$-module surjection $\CC\Big[\xxxgraph{2a}\Big]\Big/I_a\twoheadrightarrow R(\Pi_{(2^a)})$, indicating that
    \begin{align}\label{ineq:h_a[h_2]}
        \grFrob\Big(\CC\Big[\xxxgraph{2a}\Big]\Big/I_a;q\Big)\ge\grFrob(R(\Pi_{(2^a)});q).
    \end{align}
    Our goal is to force this inequality to be an equality, and thus force the containment \eqref{eq:contain_h_a[h_2]} to be an equality.

    Let $\xxx_{n\times n}$ be the variable matrix $(x_{i,j})_{1\le i,j\le n}$ (here $x_{i,j}$ is indexed by an ordered pair $(i,j)$, which is distinct from $x_{\{i,j\}}$). Let $\PM_{2a}\subseteq\Mat_{(2a)\times (2a)}(\CC)$ be the locus of all the permutation matrices corresponding to involutions without fixed points. According to \cite[Proposition 4.5]{MR4887799}, we have that
    \begin{align}\label{eq:pm_presentation}
        R(\PM_{2a}) = \CC[\xxx_{(2a)\times (2a)}]/J_a
    \end{align}
    where $J_a\subseteq\CC[\xxx_{(2a)\times (2a)}]$ is the ideal generated by
    \begin{itemize}
        \item all products $x_{i,j}\cdot x_{i,j^\prime}$ and $x_{i,j}\cdot x_{i^\prime,j}$ for $i,i^\prime,j,j^\prime\in[2a]$,
        \item all sums $\sum_{j=1}^{2a} x_{i,j}$ for $i\in[2a]$,
        \item all sums $\sum_{i=1}^{2a} x_{i,j}$ for $j\in[2a]$,
        \item all differences $x_{i,j}-x_{j,i}$ for $i,j\in[2a]$, and
        \item all diagonal variables $x_{i,i}$ for $i\in[2a]$.
    \end{itemize}
    It is clear that $\CC\Big[\xxxgraph{2a}\Big]\Big/I_a \cong \CC[\xxx_{(2a)\times(2a)}]/J_a$ as $\symm_{2a}$-modules, since the ideal $J_a$ kills diagonal variables $x_{i,i}$ and identifies $x_{i,j}$ with $x_{j,i}$. Therefore, the presentation \eqref{eq:pm_presentation} reveals that
    \[\CC\Big[\xxxgraph{2a}\Big]\Big/I_a \cong R(\PM_{2a})\]
    and hence
    \begin{align}\label{eq:h_a[h_2]_frob_equal}\grFrob\Big(\CC\Big[\xxxgraph{2a}\Big]\Big/I_a;q\Big) = \grFrob(R(\PM_{2a});q) = \sum_{\substack{\lambda\vdash 2a \\ \text{$\lambda$ is even}}} q^{\frac{2a-\lambda_1}{2}} \cdot s_{\lambda}\end{align}
    where the last equal sign arises from \cite[Theorem 4.4]{MR4887799}. Now we substitute $q=1$ in Equation~\eqref{eq:h_a[h_2]_frob_equal} and obtain
    \[\grFrob\Big(\CC\Big[\xxxgraph{2a}\Big]\Big/I_a;1\Big) = \sum_{\substack{\lambda\vdash 2a \\ \text{$\lambda$ is even}}}s_{\lambda} = h_a[h_2]=\Frob(R(\Pi_{(2^a)})) = \grFrob(R(\Pi_{(2^a)});1),\]
    which forces the inequality \eqref{ineq:h_a[h_2]} to an equality, and thus forces the containment \eqref{eq:contain_h_a[h_2]} to be an equality. Therefore, combining the equality $\gr\II(\Pi_{(2^a)}) =  I_a$ and Equation~\eqref{eq:h_a[h_2]_frob_equal} concludes the proof.
\end{proof}

With the generating set in Theorem~\ref{thm:grfrob_h_a[h_2]}, we are able to analyze the special case $b=2$ for Conjecture~\ref{conj:contain}.
\begin{theorem}\label{thm:contain_b=2}
    For all positive integers $a$, we have that
    \[\gr\II(\Pi_{(2^a)})\subseteq\gr\II(\Pi_{(a^2)}).\]
\end{theorem}
\begin{proof}
    It suffices to show that $\gr\II(\Pi_{(a^2)})$ contains the two families of generators mentioned in Theorem~\ref{thm:grfrob_h_a[h_2]}:
    \begin{itemize}
    \item all sums $\sum_{\substack{j\in[2a] \\ j\neq i}} x_{\{i,j\}}$ for $i\in[2a]$, and
        \item all products $x_S\cdot x_T$ for $S,T\in\binom{[2a]}{2}$ with $S\cap T\neq\varnothing$ ($S$ may equal $T$).
    \end{itemize}
    
    For the first type of generators $\sum_{\substack{j\in[2a] \\ j\neq i}} x_{\{i,j\}}$, we note that for any $\pi\in\Pi_{(a^2)}$ the integer $i$ is contained in exactly one block of $\pi$ and that each block of $\pi$ totally has $a$ elements. This observation indicates that
    \[\sum_{\substack{j\in[2a] \\ j\neq i}} x_{\{i,j\}} - (a-1) \in \II(\Pi_{(a^2)})\]
    and hence
    \[\sum_{\substack{j\in[2a] \\ j\neq i}} x_{\{i,j\}} \in \gr\II(\Pi_{(a^2)}).\]

    For the second type of generators $x_S\cdot x_T$ with $S\cap T\neq\varnothing$, we regard two cases depending on whether $S$ equals $T$. If $S=T$, the fact that all points in the locus $\Pi_{(a^2)}$ are $0$-$1$ vectors indicates that
    \[x_S\cdot x_T-x_S = x_S\cdot(x_S-1)\in\II(\Pi_{(a^2)}),\]
    yielding that
    \[x_S\cdot x_T\in\gr\II(\Pi_{(a^2)}).\]
    If $S\neq T$, we write $S=\{i,j\}$ and $T=\{i,k\}$ where $i,j,k\in[2a]$ are pairwise distinct. Consider
    \[f\coloneqq (x_S-x_T)^2-(1-x_{\{j,k\}})=g-h\]
    where $g=(x_S-x_T)^2,h=1-x_{\{j,k\}}$. 
    We claim that $f\in\II(\Pi_{(a^2)})$. In fact, for any unordered set partition $\pi\in\Pi_{(a^2)}$, we only need to focus on whether $j,k$ belong to the same block of $\pi$. Note that
    \[g(\pi) = \begin{cases}
        0^2 = 0, &\text{if $j,k$ belongs to the same block of $\pi$}\\
        (\pm1)^2=1, &\text{otherwise}
    \end{cases}\]
    and that
    \[h(\pi) = \begin{cases}
        1-1=0, &\text{if $j,k$ belongs to the same block of $\pi$}\\
        1-0=1, &\text{otherwise}
    \end{cases}.\]
    Consequently, we have that $f(\pi)=g(\pi)-h(\pi)=0$ for any $\pi\in\Pi_{(a^2)}$, revealing that $f\in\II(\Pi_{(a^2)})$. As a result, we deduce that
    $(x_S-x_T)^2\in\gr\II(\Pi_{(a^2)})$.
    As proved before, we have that $x_S^2,x_T^2\in\gr\II(\Pi_{(a^2)})$, from which it follows that $x_S\cdot x_T = \frac{x_S^2+x_T^2-(x_S-x_T)^2}{2}\in\gr\II(\Pi_{(a^2)})$, completing the proof.
\end{proof}
As Theorem~\ref{thm:contain_b=2} solves the special case $b=2$ of Conjecture~\ref{conj:contain}, it provides an alternative proof of the special case $h_a[h_2]\ge h_2[h_a]$ of Foulkes' conjecture. In addition, Theorem~\ref{thm:contain_b=2} yields the graded $\symm_{2a}$-module structure of $R(\Pi_{(a^2)})$, generating a graded refinement of $h_2[h_a]$ as follows.

\begin{theorem}\label{thm:grad_frob_h_2[h_a]}
    For $a\ge 1$, we have that
    \[\grFrob(R(\Pi_{(a^2)});q) = \sum_{d=0}^{\lfloor a/2\rfloor}q^d\cdot s_{(2a-2d,2d)}.\]
\end{theorem}
\begin{proof}
    Theorem~\ref{thm:contain_b=2} indicates a graded $\symm_{2a}$ surjection
    \[R(\Pi_{(2^a)})\twoheadrightarrow R(\Pi_{(a^2)}).\]
    Furthermore, Theorem~\ref{thm:grfrob_h_a[h_2]} indicates that for $d\ge 0$ any Specht module $V^\lambda$ occurring in $R(\Pi_{(2^a)})_d$ satisfies that $\lambda_1=2a-2d$. Consequently, the graded surjection above forces every Specht module $V^\lambda$ occurring in $R(\Pi_{(a^2)})_d$ to also satisfy that $\lambda_1=2a-2d$. It follows that
    \[\Frob(R(\Pi_{(a^2)})_d)=\{\Frob(R(\Pi_{(a^2)}))\}_{\lambda_1=2a-2d} = \{h_2[h_a]\}_{\lambda_1=2a-2d} = \begin{cases}
        s_{(2a-2d,2d)}, &\text{if $2a-2d \ge 2d$} \\
        0, &\text{otherwise}
    \end{cases}\]
    where the last equal sign arises from Equation~\eqref{eq:h_2[h_b]}, so the proof is complete.
\end{proof}

Interestingly, Theorem~\ref{thm:grad_frob_h_2[h_a]} means that each graded direct summand of $R(\Pi_{(a^2)})$ is irreducible as an $\symm_{2a}$-module. In other words, for $0\le d\le \lfloor a/2 \rfloor$ the graded direct summand $R(\Pi_{(a^2)})_d$ provides another model for the Specht module $V^{(2a-2d,2d)}$, which realizes $V^{(2a-2d,2d)}$ as a quotient of $\CC\Big[\xxxgraph{2a}\Big]_d$. To better understand this quotient, we find a concise generating set for the defining ideal $\gr\II(\Pi_{(a^2)})$.

\begin{proposition}\label{prop:ideal_h_2[h_a]}
    For $a\ge 1$, the defining ideal $\gr\II(\Pi_{(a^2)})$ of $R(\Pi_{(a^2)})$ is generated by
    \begin{itemize}
        \item all products $x_S\cdot x_T$ for $S,T\in\binom{[2a]}{2}$ with $S\cap T\neq\varnothing$ ($S$ may equal $T$),
        \item all sums $\sum_{\substack{j\in[2a] \\ j\neq i}} x_{\{i,j\}}$ for $i\in[2a]$, and
        \item all differences $x_{\{i,j\}}\cdot x_{\{k,l\}} - x_{\{i,k\}}\cdot x_{\{j,l\}}$ for pairwise distinct $i,j,k,l\in[2a]$.
    \end{itemize}
    Equivalently (see the generating set of $\gr\II(\Pi_{(2^a)})$ in Theorem~\ref{thm:grfrob_h_a[h_2]}), we have that
    \[\gr\II(\Pi_{(a^2)}) = \gr\II(\Pi_{(2^a)})+\langle x_{\{i,j\}}\cdot x_{\{k,l\}} - x_{\{i,k\}}\cdot x_{\{j,l\}} \,:\, \text{$i,j,k,l\in[2a]$ are pairwise distinct}\rangle,\]
    which refines Theorem~\ref{thm:contain_b=2}.
\end{proposition}
\begin{proof}
    Let $K_a\subseteq\CC\Big[\xxxgraph{2a}\Big]$ be the ideal generated by the three families of polynomials above. That is,
    \[K_a=\gr\II(\Pi_{(2^a)})+\langle x_{\{i,j\}}\cdot x_{\{k,l\}} - x_{\{i,k\}}\cdot x_{\{j,l\}} \,:\, \text{$i,j,k,l\in[2a]$ are pairwise distinct}\rangle.\]

    We first show that
    \begin{align}\label{eq:contain_h_2[h_a]}
        K_a\subseteq\gr\II(\Pi_{(a^2)}).
    \end{align}
    As we have shown that $\gr\II(\Pi_{(2^a)})\subseteq\gr\II(\Pi_{(a^2)})$ in Theorem~\ref{thm:contain_b=2}, it remains to show that
    \[x_{\{i,j\}}\cdot x_{\{k,l\}} - x_{\{i,k\}}\cdot x_{\{j,l\}}\in\gr\II(\Pi_{(a^2)})\]
    for pairwise distinct $i,j,k,l\in[2a]$. This fact immediately follows from the observation
    \[\Big(x_{\{i,j\}}-\frac{1}{2}\Big)\Big(x_{\{k,l\}}-\frac{1}{2}\Big)-\Big(x_{\{i,k\}}-\frac{1}{2}\Big)\Big(x_{\{j,l\}}-\frac{1}{2}\Big)\in\II(\Pi_{(a^2)})\]
    which can be verified by arbitrarily choosing $\pi\in\Pi_{(a^2)}$ and then considering all the three possibilities about how $\pi$ divides $\{i,j,k,l\}$: $4+0,3+1,2+2$. Therefore, the containment \eqref{eq:contain_h_2[h_a]} holds.

    Our final goal is to force the containment \eqref{eq:contain_h_2[h_a]} to be an equality. Note that \eqref{eq:contain_h_2[h_a]} naturally induces a graded $\symm_{2a}$-module surjection
    \[\CC\Big[\xxxgraph{2a}\Big]\Big/K_{a}\twoheadrightarrow R(\Pi_{(a^2)}).\]
    It suffices to force this surjection to be an isomorphism. That is, we only need to show that
    \begin{align}\label{ineq:h_2[h_a]_force}\Frob\Big(\Big(\CC\Big[\xxxgraph{2a}\Big]\Big/K_{a}\Big)_d\Big) \le \Frob(R(\Pi_{(a^2)})_d)\end{align}
    for all $d\ge 0$. A straightforward observation is that $\Big(\CC\Big[\xxxgraph{2a}\Big]\Big/K_{a}\Big)_d$ is spanned by the family of monomials modulo $K_a$
    \begin{align}\label{eq:h_2[h_a]_span}
        \Bigg\{\prod_{i=1}^d x_{S_i}\,:\,\text{$S_1,\dots,S_d\in\binom{[2a]}{2}$ are pairwise disjoint}\Bigg\},
    \end{align}
    which follows from the fact that all the elements $x_S\cdot x_T\in\gr\II(\Pi_{(2^a)})\subseteq K_a$ for $S,T\in\binom{[2a]}{2}$ with $S\cap T\neq\varnothing$ ($S$ may equal $T$) kill all monomials not belonging to \eqref{eq:h_2[h_a]_span}. We further note that $x_{\{i,j\}}\cdot x_{\{k,l\}} - x_{\{i,k\}}\cdot x_{\{j,l\}}\in K_a$ can help us identifies two monomials $x_{S_1}\cdot x_{S_2}\cdot x_{S_3}\cdots x_{S_d}$ and $x_{S_1^\prime}\cdot x_{S_2^\prime}\cdot x_{S_3}\cdots x_{S_d}$ modulo $K_a$ with $S_1\cup S_2 = S_1^\prime\cup S_2^\prime$. Iterating this identification yields
    \begin{align}\label{eq:h_2[h_a]_identification}
        \prod_{i=1}^d x_{S_i} \equiv \prod_{i=1}^d x_{T_i} \mod{K_a}
    \end{align}
    for all pairwise disjoint sets $S_1,\dots,S_d\in\binom{[2a]}{2}$ and pairwise disjoint sets $T_1,\dots,T_d\in\binom{[2a]}{2}$ such that $\bigcup_{i=1}^d S_i = \bigcup_{i=1}^d T_i$. The spanning set \eqref{eq:h_2[h_a]_span} and the identification \eqref{eq:h_2[h_a]_identification} indicates an $\symm_n$-module surjection
    \[\CC\Big[\binom{[2a]}{2d}\Big]\twoheadrightarrow \Big(\CC\Big[\xxxgraph{2a}\Big]\Big/K_{a}\Big)_d,\]
    which yields
    \begin{align}\label{ineq:force-h_2[h_a]-frob}\Frob\Big(\Big(\CC\Big[\xxxgraph{2a}\Big]\Big/K_{a}\Big)_d\Big)\le\Frob\Big(\CC\Big[\binom{[2a]}{2d}\Big]\Big) \\\nonumber=\Frob(\mathbf{1}_{\symm_{2d}}\circ\mathbf{1}_{\symm_{2a-2d}}) = h_{2d}\cdot h_{2a-2d}.\end{align}
    Furthermore, the containment $\gr\II(\Pi_{(2^a)})\subseteq K_a$ reveals that \[\Frob\Big(\Big(\CC\Big[\xxxgraph{2a}\Big]\Big/K_{a}\Big)_d\Big) \le \Frob(R(\Pi_{(2^a)})_d),\]
    combining which with Theorem~\ref{thm:grfrob_h_a[h_2]} indicates that each Specht module $V^\lambda$ occurring in the $\symm_{2a}$-module $\Big(\CC\Big[\xxxgraph{2a}\Big]\Big/K_{a}\Big)_d$ satisfies that $\lambda_1=2a-2d$. This $\lambda_1$-restriction together with Inequality~\eqref{ineq:force-h_2[h_a]-frob} means that
    \[\Frob\Big(\Big(\CC\Big[\xxxgraph{2a}\Big]\Big/K_{a}\Big)_d\Big) \le \{h_{2d}\cdot h_{2a-2d}\}_{\lambda_1=2a-2d}=\begin{cases}
        s_{(2a-2d,2d)}, &\text{if $2d\le 2a-2d$} \\
        0, &\text{otherwise}
    \end{cases}\]
    where the last equal sign arises from the Pieri rule for Schur functions. Therefore, we finish the proof of \eqref{ineq:h_2[h_a]_force}, concluding the whole proof.
\end{proof}

\subsection{First-row length separation phenomena}\label{subsec:length-generalize}
In Subsection~\ref{subsec:b=2}, both $R(\Pi_{(2^a)})$ and $R(\Pi_{(a^2)})$ follows ``first-row length separation phenomena''. Specifically, $\grFrob(R(\Pi_{(2^a)});q)$ separates the Schur expansion of $h_a[h_2]$ according to $\lambda_1$ (see Theorem~\ref{thm:grfrob_h_a[h_2]}), and $\grFrob(R(\Pi_{(a^2)});q)$ separates the Schur expansion of $h_2[h_a]$ according to $\lambda_1$ as well (see Theorem~\ref{thm:grad_frob_h_2[h_a]}). However, both $R(\Pi_{(2^a)})$ and $R(\Pi_{(a^2)})$ are quotients of $\CC\Big[\xxxgraph{2a}\Big]$. Now we want to extend their construction to $\CC\Big[\xxxgraph{n}\Big]$ where $n$ may be odd. We use the same defining ideals and will see that ``first-row length separation phenomena'' still hold (see Theorems~\ref{thm:generalize-h_a[h_2]} and \ref{thm:generalize-h_2[h_a]}).

\begin{definition}\label{def:IJ}
    For integers $n>0$, $I(n)\subseteq\CC\Big[\xxxgraph{n}\Big]$ is the ideal generated by
    \begin{itemize}
        \item all sums $\sum_{\substack{j\in[n]\\j\neq i}}x_{\{i,j\}}$ for $i\in[n]$, and
        \item all products $x_S\cdot x_T$ for $S,T\in\binom{[n]}{2}$ with $S\cap T\neq\varnothing$ ($S$ may equal $T$).
    \end{itemize}
    Let $J(n)\subseteq\CC\Big[\xxxgraph{n}\Big]$ be the ideal given by
    \[J(n)\coloneqq I(n) + \langle x_{\{i,j\}}\cdot x_{\{k,l\}} - x_{\{i,k\}}\cdot x_{\{j,l\}} \,:\,\text{$i,j,k,l\in[n]$ are pairwise distinct} \rangle.\]
\end{definition}

We will need strategic spanning sets of both quotients $\CC\Big[\xxxgraph{n}\Big]\Big/I(n)$ and $\CC\Big[\xxxgraph{n}\Big]\Big/J(n)$ to study their $\symm_n$-module structure. Our spanning sets consist of the following monomials.

\begin{definition}\label{def:mono}
    Let $\MMM_{n,d}$ be the set of \emph{$d$-matchings} given by
    \[\MMM_{n,d}\coloneqq\bigg\{\{S_1,\dots,S_d\}\,:\,\text{$S_1,\dots,S_d\in\binom{[n]}{2}$ are pairwise disjoint}\bigg\}.\]
    For $\tau\in\MMM_{n,d}$, we define the \emph{matching monomial} $\mmm(\tau)\in\CC\Big[\xxxgraph{n}\Big]$ by
    \[\mmm(\tau)\coloneqq\prod_{S\in\tau} x_{S}.\]
    For $S\in\binom{[n]}{2d}$, we write $S=\{i_1<i_2<\dots<i_{2d}\}$ and define the \emph{standard matching monomial} $\mmm(S)\in\CC\Big[\xxxgraph{n}\Big]$ by
    \[\mmm(S)\coloneqq\mmm(\tau)\]
    where $\tau=\{\{i_1,i_2\},\{i_3,i_4\},\dots,\{i_{2d-1},i_{2d}\}\}\in\MMM_{n,d}$. In particular, we have $\mmm(\varnothing) = 1$.
\end{definition}

\begin{example}\label{ex:mono}
    Let $\tau\in\MMM_{7,3}$ be the $3$-matching $\{\{2,4\},\{5,7\},\{3,6\}\}$. Then we have the matching monomial $\mmm(\tau)= x_{\{2,4\}}\cdot x_{\{5,7\}}\cdot x_{\{3,6\}}$. Let $S=\{2,3,4,5,6,7\}$. Then we have the standard matching monomial $\mmm(S) = x_{\{2,3\}}\cdot x_{\{4,5\}}\cdot x_{\{6,7\}}$.
\end{example}
The quotient $\CC\Big[\xxxgraph{n}\Big]\Big/I(n)$ possesses the following spanning set.
\begin{lemma}\label{lem:span-I}
    The family of matching monomials
    \[\bigsqcup_{d=0}^{\lfloor n/2\rfloor}\{\mmm(\tau)\,:\,\tau\in\MMM_{n,d}\}\]
    descends to a spanning set of $\CC\Big[\xxxgraph{n}\Big]\Big/I(n)$.
\end{lemma}
\begin{proof}
    This is because all other monomials are killed by the products $x_S\cdot x_T\in I(n)$ with $S\cap T\neq\varnothing$.
\end{proof}
For $\CC\Big[\xxxgraph{n}\Big]\Big/J(n)$, the spanning set above can be narrowed down.
\begin{corollary}\label{cor:span-J}
    The family of standard matching monomials
    \[\bigsqcup_{d=0}^{\lfloor n/2 \rfloor}\bigg\{\mmm(S)\,:\, S\in\binom{[n]}{2d}\bigg\}\]
    descends to a spanning set of $\CC\Big[\xxxgraph{n}\Big]\Big/J(n)$. Furthermore, we have that
    \[w\cdot\mmm(S)\equiv\mmm(w\cdot S)\mod{J(n)}\]
    for all $w\in\symm_n$.
\end{corollary}
\begin{proof}
    Since $I(n)\subseteq J(n)$, Lemma~\ref{lem:span-I} indicates that the family of matching monomials
    \begin{align}\label{eq:generalize-span-coarse}\bigsqcup_{d=0}^{\lfloor n/2\rfloor}\{\mmm(\tau)\,:\,\tau\in\MMM_{n,d}\}\end{align}
    also descends to a spanning set of $\CC\Big[\xxxgraph{n}\Big]\Big/J(n)$. However, a lot of these matching monomials are identified with each other modulo $J(n)$, because all the differences $x_{\{i,j\}}\cdot x_{\{k,l\}} - x_{\{i,k\}}\cdot x_{\{j,l\}}$ allows us to ``interchange'' variable indices by
    \[\mmm(\tau)\equiv\mmm(\tau^\prime)\mod{J(n)}\]
    for $\tau=\{S_1,S_2,S_3,\dots,S_d\},\tau^\prime=\{S_1^\prime,S_2^\prime,S_3,\dots,S_d\}$ with $S_1\cup S_2 = S_1^\prime\cup S_2^\prime$. Iterating such interchanges, we obtain the identification
    \begin{align}\label{eq:generalize-identification}\mmm(\tau)\equiv\mmm(\tau^\prime)\mod{J(n)}\end{align}
    for all $\tau,\tau^\prime \in \MMM_{n,d}$ with $\bigcup_{S\in\tau }S = \bigcup_{T\in\tau^\prime} T$. In particular, it follows that any matching monomial $\mmm(\tau)$ is identified with the standard matching monomial $\mmm(\bigcup_{S\in\tau}S)$ modulo $J(n)$. Therefore, the spanning set \eqref{eq:generalize-span-coarse} can be replaced with a smaller one
    \[\bigsqcup_{d=0}^{\lfloor n/2 \rfloor}\bigg\{\mmm(S)\,:\, S\in\binom{[n]}{2d}\bigg\}.\]
    This set modulo $J(n)$ is stable under the $\symm_n$-action by
    \[w\cdot\mmm(S)=w\cdot\mmm(\tau)=\mmm(w\cdot\tau)\equiv\mmm(w\cdot S)\mod{J(n)}\]
    where $S=\{i_1<i_2<\dots<i_{2d}\},\tau=\{\{i_{2k-1},i_{2k}\}\,:\,1\le k\le d\}$, the congruence $\equiv$ of which arises from \eqref{eq:generalize-identification}. Now our proof is complete.
\end{proof}

The result below enables us to govern the length $\lambda_1$ of the first row of $\lambda$ for all Specht modules $V^\lambda$ occurring in $\Big(\CC\Big[\xxxgraph{n}\Big]\Big/I(n)\Big)_d$.

\begin{lemma}\label{lem:generalize-length}
    For any $\tau\in\MMM_{n,d}$, we have the annihilation property
    \[\eta_j\cdot\mmm(\tau)\in I(n)\]
    for $n-2d<j\le n$.
\end{lemma}
\begin{proof}
    Write $\tau=\{\{i_1,j_1\},\{i_2,j_2\},\dots,\{i_d,j_d\}\}$ and write $S=\{i_1,j_1,i_2,j_2,\dots,i_d,j_d\}$. Suppose that $S\cap[j]=(\bigcup_{k=1}^t\{i_k,j_k\})\cup(\bigcup_{k=t+1}^{t+s}\{i_k\})$ and denote this intersection by $T$. Write $m=\mmm(\tau)/\prod_{k=1}^{t+s}x_{\{i_k,j_k\}}$. Let $\doteq$ denote equalities up to a nonzero multiple. We have that
    \begin{align*}
        &\eta_j \cdot \mmm(\tau) \doteq \sum_{f:T\hookrightarrow[j]}\Bigg(\prod_{k=1}^t x_{\{f(i_k),f(j_k)\}}\Bigg)\Bigg(\prod_{k=t+1}^{t+s}x_{\{f(i_k),j_k\}}\Bigg)m \\
        &\equiv  \Bigg(\prod_{k=1}^t \sum_{f_k:\{i_k,j_k\}\hookrightarrow[j]} x_{\{f_k(i_k),f_k(j_k)\}}\Bigg)\Bigg(\prod_{k=t+1}^{t+s} \sum_{f_k:\{i_k\}\rightarrow [j]} x_{\{f(i_k),j_k\}}\Bigg)m\\
        &\equiv \Bigg(\prod_{k=1}^t \Bigg(-\sum_{f_k:\{i_k,j_k\}\hookrightarrow[n]\setminus[j]} x_{\{f_k(i_k),f_k(j_k)\}}\Bigg)\Bigg)\Bigg(\prod_{k=t+1}^{t+s}\Bigg(-\sum_{f_k:\{i_k\}\rightarrow [n]\setminus([j]\cup\{j_k\})} x_{\{f(i_k),j_k\}}\Bigg)\Bigg)m \\
        &\doteq \Bigg(\prod_{k=1}^t \sum_{f_k:\{i_k,j_k\}\hookrightarrow[n]\setminus[j]} x_{\{f_k(i_k),f_k(j_k)\}}\Bigg)\Bigg(\prod_{k=t+1}^{t+s}\sum_{f_k:\{i_k\}\rightarrow [n]\setminus([j]\cup\{j_k\})} x_{\{f(i_k),j_k\}}\Bigg)m \\
        &\equiv \sum_{f:T\hookrightarrow ([n]\setminus[j])\setminus(S\setminus T)}\Bigg(\prod_{k=1}^t x_{\{f(i_k),f(j_k)\}}\Bigg)\Bigg(\prod_{k=t+1}^{t+s}x_{\{f(i_k),j_k\}}\Bigg)m \mod{I(n)}
    \end{align*}
    where the first and the last congruences $\equiv$ arise from the fact that $x_S\cdot x_T\in I(n)$ whenever $S\cap T\neq\varnothing$, and the second congruence $\equiv$ arises from the fact that $\sum_{\substack{j\in[n]\\j\neq i}}x_{\{i,j\}}\in I(n)$ for all $i\in[n]$. Nonetheless, we note that
    \[\lvert([n]\setminus[j])\setminus(S\setminus T)\rvert = (n-j)-(\lvert S \rvert - \lvert T \rvert) = n-j-2d+\lvert T\rvert<\lvert T \rvert,\]
    which means that there are no injections $f:T\hookrightarrow([n]\setminus[j])\setminus(S\setminus T)$. Therefore, the chain of equations above indicates that $\eta_j\cdot\mmm(\tau)\in I(n)$, finishing the proof.
\end{proof}

\begin{corollary}\label{cor:generalize-length}
    For $0\le d\le\lfloor n/2\rfloor$, any Specht module $V^\lambda$ occurring in either $\Big(\CC\Big[\xxxgraph{n}\Big]\Big/I(n)\Big)_d$ or $\Big(\CC\Big[\xxxgraph{n}\Big]\Big/J(n)\Big)_d$ satisfies that $\lambda_1\le n-2d$.
\end{corollary}
\begin{proof}
    For $\Big(\CC\Big[\xxxgraph{n}\Big]\Big/I(n)\Big)_d$, this result is immediate from Lemmas~\ref{lem:row-length}, \ref{lem:span-I} and \ref{lem:generalize-length}. Then we deduce similar result for $\Big(\CC\Big[\xxxgraph{n}\Big]\Big/J(n)\Big)_d$ from the containment $I(n)\subseteq J(n)$.
\end{proof}

Now we show the graded $\symm_n$-module structures. The following result generalizes Theorem~\ref{thm:grfrob_h_a[h_2]}.

\begin{theorem}\label{thm:generalize-h_a[h_2]}
    For any integer $n>0$, we have that
    \[\grFrob(\CC\Big[\xxxgraph{n}\Big]\Big/I(n);q) = \sum_{\substack{\lambda\vdash n \\ \text{$2\mid\lambda_i$ for $i>1$}}} q^{\frac{n-\lambda_1}{2}} \cdot s_\lambda.\]
\end{theorem}
\begin{proof}
    If $n$ is even, Theorem~\ref{thm:grfrob_h_a[h_2]} straightly indicates this result. From now on, we suppose that $n$ is odd. We may further suppose that $n>1$ because the case $n=1$ is trivial.
    Lemma~\ref{lem:span-I} yields $\symm_n$-module surjections
    \begin{align*}
        \CC[\MMM_{n,d}] &\twoheadrightarrow \Big(\CC\Big[\xxxgraph{n}\Big]\Big/I(n)\Big)_d \\
        \tau&\mapsto \mmm(\tau)\mod{I(n)}\qquad\text{for all $\tau\in\MMM_{n,d}$}
    \end{align*}
    for $0\le d\le\lfloor n/2 \rfloor$. Therefore, we have that
    \begin{align}\label{ineq:generalize-h_a[h_2]}
        &\Frob\Big(\Big(\CC\Big[\xxxgraph{n}\Big]\Big/I(n)\Big)_d\Big) \le \Frob(\CC[\MMM_{n,d}]) = \Frob(\CC[\Pi_{(2^d)}]\circ\mathbf{1}_{\symm_{n-2d}}) \\ \nonumber=& h_d[h_2]\cdot h_{n-2d} =\sum_{\substack{\mu\vdash 2d\\\text{$\mu$ is even}}} s_\mu \cdot h_{n-2d} = \sum_{\substack{\mu\vdash 2d\\\text{$\mu$ is even}}}\sum_{\substack{\text{$\lambda/\mu$ is a horizontal stripe,}\\ \lvert\lambda/\mu\rvert = n-2d}} s_\lambda
    \end{align}
    where the last equal sign uses the Pieri rule. Then we apply Corollary~\ref{cor:generalize-length} to truncate the last expression in \eqref{ineq:generalize-h_a[h_2]}, yielding that
    \begin{align}\label{ineq:generalize-h_a[h_2]-strong}
        \Frob\Big(\Big(\CC\Big[\xxxgraph{n}\Big]\Big/I(n)\Big)_d\Big)  \le \Bigg\{\sum_{\substack{\mu\vdash 2d\\\text{$\mu$ is even}}}\sum_{\substack{\text{$\lambda/\mu$ is a horizontal stripe,}\\ \lvert\lambda/\mu\rvert = n-2d}} s_\lambda \Bigg\}_{\lambda_1\le n-2d} = \sum_{\substack{\lambda\vdash n \\ \lambda_1=n-2d \\ \text{$2\mid\lambda_i$ for $i>1$}}} s_\lambda.
    \end{align}
    for $0\le d\le\lfloor n/2 \rfloor$. Our goal is to force this inequality to be an equality.

    Consider the graded surjection
    \begin{align*}
        \mathcal{F}\,:\,\CC\Big[\xxxgraph{n}\Big] &\twoheadrightarrow \CC\Big[\xxxgraph{n-1}\Big]
    \end{align*}
    given by the evaluation $x_{\{i,n\}} = 0$ for all $i\in[n-1]$. Note that $\mathcal{F}$ is $\symm_{n-1}$-equivariant and that $\mathcal{F}(I(n))\subseteq I(n-1)$. Then $\mathcal{F}$ descends to a graded $\symm_{n-1}$-module surjection
    \[\Res_{\symm_{n-1}}^{\symm_n}\Big(\CC\Big[\xxxgraph{n}\Big]\Big/I(n)\Big) \twoheadrightarrow \CC\Big[\xxxgraph{n-1}\Big]\Big/I(n-1)=R\Big(\Pi_{(2^{\frac{n-1}{2}})}\Big),\]
    indicating that
    \[\Frob\Big(\Res_{\symm_{n-1}}^{\symm_n}\Big(\CC\Big[\xxxgraph{n}\Big]\Big/I(n)\Big)_d\Big) \ge \Frob\Big(R\Big(\Pi_{(2^{\frac{n-1}{2}})}\Big)_d\Big) = \sum_{\substack{ \lambda\vdash n-1\\ \lambda_1 =n-1-2d \\ \text{$\lambda$ is even}}}s_\lambda\]
    where the two equal sign follows from Theorem~\ref{thm:grfrob_h_a[h_2]}.
    Combining the inequality above with the Branching rule, we force \eqref{ineq:generalize-h_a[h_2]-strong} to be an equality for $0\le d\le\lfloor n/2\rfloor$. As Lemma~\ref{lem:span-I} indicates that $\Big(\CC\Big[\xxxgraph{n}\Big]\Big/I(n)\Big)_d=0$ for $d>\lfloor n/2\rfloor$, the proof is complete.
\end{proof}

The following result generalizes Theorem~\ref{thm:grad_frob_h_2[h_a]}.

\begin{theorem}\label{thm:generalize-h_2[h_a]}
    For any integer $n>0$, we have that
    \[\grFrob\Big(\CC\Big[\xxxgraph{n}\Big]\Big/J(n);q\Big) = \sum_{d=0}^{\lfloor n/4 \rfloor} q^d \cdot s_{(n-2d,2d)}.\]
\end{theorem}
\begin{remark}\label{rmk:generalize-h_2[h_a]}
    Theorem~\ref{thm:generalize-h_2[h_a]} provides the Specht module $V^{(n-2d,2d)}$ with another model, which is a quotient of $\CC\Big[\xxxgraph{n}\Big]_d$.
\end{remark}
\begin{proof}
    If $n$ is even, Proposition~\ref{prop:ideal_h_2[h_a]} indicates that $\CC\Big[\xxxgraph{n}\Big]\Big/J(n) = R(\Pi_{((n/2)^2)})$, and hence Theorem~\ref{thm:grad_frob_h_2[h_a]} concludes our proof. Therefore, we suppose that $n$ is odd. We further require that $n>1$ because the case $n=1$ is trivial. Then Corollary~\ref{cor:span-J} indicates $\symm_n$-module surjections
    \begin{align*}
        \CC\Big[\binom{[n]}{2d}\Big]&\twoheadrightarrow\Big(\CC\Big[\xxxgraph{n}\Big]\Big/J(n)\Big)_d \\
        S &\mapsto \mmm(S) \mod{J(n)} \qquad\text{for all $S\in\binom{[n]}{2d}$}
    \end{align*}
    for $d\ge 0$, with the convention that for $d>\lfloor n/2\rfloor$ both $\symm_n$-modules above equal $0$. Consequently, we have that
    \begin{align*}
        \Frob\Big(\Big(\CC\Big[\xxxgraph{n}\Big]\Big/J(n)\Big)_d\Big) \le \Frob\Big(\CC\Big[\binom{[n]}{2d}\Big]\Big) = \Frob(\mathbf{1}_{\symm_{2d}}\circ \mathbf{1}_{\symm_{n-2d}})=h_d\cdot h_{n-2d},
    \end{align*}
    which, together with Corollary~\ref{cor:generalize-length}, indicates that
    \begin{align}\label{ineq:generalize-h_2[h_a]}
        \Frob\Big(\Big(\CC\Big[\xxxgraph{n}\Big]\Big/J(n)\Big)_d\Big) \le \{h_d\cdot h_{n-2d}\}_{\lambda_1\le n-2d}=\begin{cases}
            s_{(n-2d,2d)}, &\text{if $n-2d\ge 2d$} \\
            0, &\text{otherwise}
        \end{cases}
    \end{align}
    where the equal sign is derived from the Pieri rule.
    Our goal is to force this inequality to be an equality.

    Consider the graded surjection
    \[\mathcal{F}\,:\,\CC\Big[\xxxgraph{n}\Big]\twoheadrightarrow\CC\Big[\xxxgraph{n-1}\Big]\]
    given by the evaluation $x_{\{i,n\}}=0$ for all $i\in[n-1]$. Note that $\mathcal{F}$ is $\symm_{n-1}$-equivariant and that $\mathcal{F}(J(n))\subseteq J(n-1)$. Thus, $\mathcal{F}$ descends to a graded $\symm_{n-1}$-module surjection
    \[\Res_{\symm_{n-1}}^{\symm_n}\Big(\CC\Big[\xxxgraph{n}\Big]\Big/J(n) \Big)\twoheadrightarrow\CC\Big[\xxxgraph{n-1}\Big]\Big/J(n-1) = R\Big(\Pi_{((\frac{n-1}{2})^2)}\Big)\]
    where the equal sign follows from Proposition~\ref{prop:ideal_h_2[h_a]}. We thus deduce that
    \[\Frob\Big(\Res_{\symm_{n-1}}^{\symm_{n}}\Big(\CC\Big[\xxxgraph{n}\Big]\Big/J(n)\Big)_d\Big)\ge\Frob\Big(R\Big(\Pi_{((\frac{n-1}{2})^2)}\Big)_d\Big)=\begin{cases}
        s_{(n-1-2d,2d)}, &\text{if $0\le d\le\lfloor\frac{n-1}{4}\rfloor$} \\
        0, &\text{otherwise}
    \end{cases}\]
    where the equal sign arises from Theorem~\ref{thm:grad_frob_h_2[h_a]}. This inequality, together with the Branching rule, forces \eqref{ineq:generalize-h_2[h_a]} to be an equality (note that for odd integer $n$ we have that $n-2d\ge 2d$ if and only if $n-1-2d\ge 2d$), which finishes the proof.
\end{proof}

\subsection{General loci $\Pi_\lambda$ and their unions}\label{subsec:general}

For $0\le m\le n$, we consider a larger locus
\[\Pi_{n,m}\coloneqq\bigsqcup_{\substack{\lambda\vdash n \\ \lambda_1\le m}} \Pi_\lambda.\]
Interestingly, we find the standard monomial basis of $R(\Pi_{n,m})$ with respect to any monomial order by Theorem~\ref{thm:general-satandard}. We also give a concise presentation for the ring $R(\Pi_{n,m})$ in Corollary~\ref{cor:basis-presentation}. Here are some useful definitions and technical results beforehand.

\begin{definition}\label{def:graph-mono}
    For any undirected simple graph $G$ with vertex set $[n]$ and edge set $E(G)\subseteq\binom{[n]}{2}$, we define the \emph{graph monomial} $\mmm(G)\in\CC\Big[\xxxgraph{n}\Big]$ by
    \[\mmm(G)\coloneqq\prod_{S\in E(G)}x_S.\]
\end{definition}
\begin{remark}\label{rmk:graph-mono}
    Note that all graph monomials are exactly all monomials without repeated variables.
\end{remark}

\begin{definition}\label{def:general-ideal}
    For $0\le m \le n$, let $I_{n,m}\subseteq\CC\Big[\xxxgraph{n}\Big]$ be the ideal generated by
    \begin{itemize}
        \item all squares $x_S^2$ for $S\in\binom{[n]}{2}$,
        \item all differences $x_{\{i,j\}}\cdot x_{\{j,k\}} - x_{\{i,k\}}\cdot x_{\{j,k\}}$ for pairwise distinct $i,j,k\in[n]$, and
        \item all monomials $\mmm(T)$ for trees $T$ with $m+1$ vertices in $[n]$.
    \end{itemize}
\end{definition}

We have the following ideal containment, and we will see that the equality holds by Corollary~\ref{cor:basis-presentation}.
\begin{lemma}\label{lem:general-contain}
    We have $I_{n,m}\subseteq\gr\II(\Pi_{n,m})$.
\end{lemma}
\begin{proof}
    It suffices to show that all three families of generators mentioned in Definition~\ref{def:general-ideal} belong to $\gr\II(\Pi_{n,m})$.
    
    For the squares $x_S^2$, we have that $x_S^2-x_S\in\II(\Pi_{n,m})$ since any unordered set partition is identified with a $0$-$1$ vector in $\CC^{\binom{[n]}{2}}$, from which we deduce that $x_S^2\in\gr\II(\Pi_{n,m})$.
    
    For the differences $x_{\{i,j\}}\cdot x_{\{j,k\}} - x_{\{i,k\}}\cdot x_{\{j,k\}}$, write $f=x_{\{i,j\}}\cdot x_{\{j,k\}}$ and $g=x_{\{i,k\}}\cdot x_{\{j,k\}}$. Note that for any $\pi\in\Pi_{n,m}$ we have
    \[f(\pi) = \begin{cases}
        1, &\text{if $i,j,k$ belong to the same block of $\pi$} \\
        0, &\text{otherwise}
    \end{cases}\]
    and
    \[g(\pi) = \begin{cases}
        1, &\text{if $i,j,k$ belong to the same block of $\pi$} \\
        0, &\text{otherwise.}
    \end{cases}\]
    Therefore, $(f-g)(\pi)=0$ for all $\pi\in\Pi_{n,m}$, so $f-g\in\II(\Pi_{n,m})$ and thus $f-g\in\gr\II(\Pi_{n,m})$.

    For the monomials $\mmm(T)$ where $T$ is a tree with $m+1$ vertices in $[n]$, we claim that $\mmm(T)\in\II(\Pi_{n,m})$. In fact, for all $\pi\in\Pi_{n,m}$, each block of $\pi$ has no more than $m$ elements, but for any unordered set partition $\pi^\prime$ we have that $\mmm(T)(\pi^\prime)\neq0$ if and only if the vertex set of $T$ is contained in a single block of $\pi^\prime$. Therefore, $\mmm(T)(\pi)$ vanishes for all $\pi\in\Pi_{n,m}$, indicating that $\mmm(T)\in\gr\II(\Pi_{n,m})$.
\end{proof}

To better understand $I_{n,m}$ and thus $\gr\II(\Pi_{n,m})$, we seek some useful elements in $I_{n,m}$, which provides us with powerful relations on the quotients $\CC\Big[\xxxgraph{n}\Big]\Big/ I_{n,m}$ and hence $R(\Pi_{n,m})$.

\begin{definition}\label{def:forest-partition}
    For a forest $F$ on the vertex set $[n]$, write all vertex sets of its connected components by $B_1,\dots,B_\ell$. Define the \emph{unordered set partition associated with $F$} by
    \[\tilde{\pi}(F)\coloneqq\{B_1,\dots,B_\ell\}\]
    which is an unordered set partition of $[n]$.
\end{definition}
\begin{example}\label{ex:forest-partition}
    For $n=7$, let $F$ be the forest on $[n]$ with edge set \[E(F)=\{\{1,2\},\{1,3\},\{1,4\},\{6,7\}\}.\]
    Then $F$ has four connected components corresponding to vertex sets $\{1,2,3,4\},\{5\},\{6,7\}$. Therefore, we have $\tilde{\pi}(F)=\{\{1,2,3,4\},\{5\},\{6,7\}\}$.
\end{example}

\begin{lemma}\label{lem:diff-two-forests}
    For two forests $F_1,F_2$ on the vertex set $[n]$ with $\tilde{\pi}(F_1)=\tilde{\pi}(F_2)$, we have that
    \[\mmm(F_1)\equiv\mmm(F_2)\mod{I_{n,m}}.\]
\end{lemma}
\begin{proof}
    We first focus on two trees $T_1,T_2$ on a vertex set $B\subseteq[n]$. Then the generators $x_{\{i,j\}}\cdot x_{\{j,k\}}-x_{\{i,k\}}\cdot x_{\{j,k\}}\in I_{n,m}$ enables us to move one edge of a tree $T$ with vertex set $B$ without changing the congruence class $\mmm(T)\mod{I_{n,m}}$ or the vertex set $B$. Therefore, we can do finitely many such moving operations to convert $T_1$ into $T_2$, indicating that $\mmm(T_1)\equiv\mmm(T_2)\mod{I_{n,m}}$. Applying this fact to all the connected components of $F_1$ and $F_2$ (where we need the assumption $\tilde{\pi}(F_1)=\tilde{\pi}(F_2)$), we conclude that $\mmm(F_1)\equiv\mmm(F_2)\mod{I_{n,m}}$.
\end{proof}

\begin{corollary}\label{cor:cycle-vanish}
    For any cycle $C$ on the vertex set $[n]$, we have that
    \[\mmm(C)\equiv 0\mod{I_{n,m}}.\]
\end{corollary}
\begin{proof}
    Write all the edges of $C$ by $\{i_1,i_2\},\{i_2,i_3\},\dots,\{i_{c-1},i_c\},\{i_c,i_1\}$. Let $P$ be the path with totally $c-1$ edges $\{i_1,i_2\},\{i_2,i_3\},\dots,\{i_{c-1},i_c\}$. Let $P^\prime$ be another path with totally $c-1$ edges $\{i_c,i_1\},\{i_1,i_2\},\dots,\{i_{c-2},i_{c-1}\}$. Lemma~\ref{lem:diff-two-forests} indicates that $\mmm(P)\equiv\mmm(P^\prime)\mod{I_{n,m}}$ and hence that
    \[\mmm(C)=\mmm(P)\cdot x_{\{i_c,i_1\}}\equiv\mmm(P^\prime)\cdot x_{\{i_c,i_1\}}\equiv 0\mod{I_{n,m}}\]
    where the last congruence $\equiv$ follows from the fact that $x_{\{i_c,i_1\}}^2\mid(\mmm(P^\prime)\cdot x_{\{i_c,i_1\}})$.
\end{proof}

This is the last technique result in this section.

\begin{lemma}\label{lem:general-span}
    For any family of forests $\{F_{\pi}\}_{\pi\in\Pi_{n,m}}$ on the vertex set $[n]$ with $\tilde{\pi}(F_{\pi}) = \pi$ for all $\pi\in\Pi_{n,m}$, the family of monomials $\{\mmm(F_\pi)\,:\pi\in\Pi_{n,m}\}$ descends to a spanning set of $\CC\Big[\xxxgraph{n}\Big]\Big/ I_{n,m}$.
\end{lemma}
\begin{proof}
    Since $x_S^2\in I_{n,m}$ for all $S\in\binom{[n]}{2}$, all the monomials with repeated variables are killed by $I_{n,m}$. Therefore, the collection of all graph monomials $\mmm(G)$ descends to a spanning set of $\CC\Big[\xxxgraph{n}\Big]\Big/ I_{n,m}$. Whenever $G$ contains a cycle as a subgraph, Corollary~\ref{cor:cycle-vanish} indicates that $\mmm(G)$ vanishes modulo $I_{n,m}$. Consequently, the family of monomials $\{\mmm(F)\,:\,\text{$F$ is a forest on $[n]$}\}$ descends to a spanning set of $\CC\Big[\xxxgraph{n}\Big]\Big/ I_{n,m}$.

    It suffices to show that for any forest $F$ on the vertex set $[n]$ we can write $\mmm(F)\mod{I_{n,m}}$ as a linear combination of the congruence classes $\mmm(F_\pi)\mod{I_{n,m}}$ with $\pi\in\Pi_{n,m}$. We prove it by regarding two cases. If $\tilde{\pi}(F)\not\in\Pi_{n,m}$, then $F$ must contain a connected component with more than $m$ vertices, so there exists a subtree $T$ of $F$ with $m+1$ vertices. Then the fact that $\mmm(T)\mid\mmm(F)$ indicates that $\mmm(F)\in I_{n,m}$. If $\tilde{\pi}(F)\in\Pi_{n,m}$, Lemma~\ref{lem:diff-two-forests} yields that $\mmm(F)\equiv\mmm(F_{\tilde{\pi}(F)})\mod{I_{n,m}}$, concluding our proof.
\end{proof}

Now we obtain lots of bases of $\CC\Big[\xxxgraph{n}\Big]\Big/ I_{n,m}$ and the presentation $R(\Pi_{n,m})=\CC\Big[\xxxgraph{n}\Big]\Big/ I_{n,m}$:
\begin{corollary}\label{cor:basis-presentation}
    For any family of forests $\{F_\pi\}_{\pi\in\Pi_{n,m}}$ on the vertex set $[n]$ with $\tilde{\pi}(F_\pi) = \pi$ for all $\pi\in\Pi_{n,m}$, the family of monomials $\{\mmm(F_\pi)\,:\,\pi\in\Pi_{n,m}\}$ descends to a basis of $R(\Pi_{n,m})$. Furthermore, we have
    \[\gr\II(\Pi_{n,m}) = I_{n,m}.\]
\end{corollary}
\begin{proof}
    Lemma~\ref{lem:general-span} indicates that the family of monomials $\{\mmm(F_\pi)\,:\,\pi\in\Pi_{n,m}\}$ descends to a spanning set of $\CC\Big[\xxxgraph{n}\Big]\Big/ I_{n,m}$, and Lemma~\ref{lem:general-contain} reveals the containment $I_{n,m}\subseteq\gr\II(\Pi_{n,m})$. Therefore, $\{\mmm(F_\pi)\,:\,\pi\in\Pi_{n,m}\}$ also descends to a spanning set of $R(\Pi_{n,m})$. However, the cardinality of this spanning set equals $\lvert\Pi_{n,m}\rvert=\dim_\CC(R(\Pi_{n,m}))$, which forces this spanning set to be a basis and also forces the containment $I_{n,m}\subseteq\gr\II(\Pi_{n,m})$ to be an equality.
\end{proof}

\begin{theorem}\label{thm:general-satandard}
    Fix a monomial order $<$ on $\CC\Big[\xxxgraph{n}\Big]$. For any $B\subseteq[n]$, let $T_B$ be the tree with vertex set $B$ such that \[\mmm(T_B) = \min\{\mmm(T)\,:\,\text{$T$ has vertex set $B$}\}.\]
    Then, the family of monomials
    \[\bigg\{\prod_{B\in\pi}\mmm(T_B)\,:\pi\in\Pi_{n,m}\bigg\}\]
    descends to the standard monomial basis of $R(\Pi_{n,m})=\CC\Big[\xxxgraph{n}\Big]\Big/I_{n,m}$.
\end{theorem}
\begin{example}\label{ex:general-standard}
    Consider the variable order given by: $x_S<x_T$ if and only if $S=\{i<j\},T=\{k<l\}$ and
    \begin{itemize}
        \item either $i<k$, or
        \item $i=k$ and $j<l$.
    \end{itemize}
    Take the monomial order $<$ to be the lexicographic order. Then, for any nonempty set $B=\{i_1<i_2<\dots<i_b\}\subseteq[n]$ we have $\mmm(T_B)=\prod_{k=2}^bx_{\{i_1,i_k\}}$. According to Theorem~\ref{thm:general-satandard}, we can use some products of these monomials $\mmm(T_B)$ to construct the standard monomial basis of $R(\Pi_{n,m})$.
\end{example}
\begin{proof}[Proof of Theorem~\ref{thm:general-satandard}]
    For convenience, denote the standard monomial basis by $\mathcal{SB}$. We claim that
    \[\mathcal{SB}\subseteq\bigg\{\prod_{B\in\pi}\mmm(T_B)\,:\pi\in\Pi_{n,m}\bigg\}.\]
    In fact, it suffices to show the reverse inclusion for the complements. Consider any monomial $\tilde{\mmm}\notin\big\{\prod_{B\in\pi}\mmm(T_B)\,:\pi\in\Pi_{n,m}\big\}$. If $\tilde{\mmm}$ has repeated variables, then $\tilde{\mmm}$ vanishes modulo $I_{n,m}$ by definition, revealing that $\tilde{\mmm}\notin\mathcal{SB}$. Now suppose that $\tilde{\mmm}$ has no repeated variables, which means that $\tilde{\mmm}=\mmm(G)$ for some simple graph $G$ on $[n]$. If $G$ is not a forest, Corollary~\ref{cor:cycle-vanish} indicates that $\tilde{\mmm}$ vanishes modulo $I_{n,m}$ and hence that $\tilde{\mmm}\notin\mathcal{SB}$. We hence suppose further that $\tilde{\mmm}=\mmm(F)$ for some forest $F$ on $[n]$. Then, Lemma~\ref{lem:diff-two-forests} indicates that
    \[\tilde{\mmm}-\prod_{B\in\tilde{\pi}(F)} \mmm(T_B)\in I_{n,m}.\]
    However, by construction we know that $\tilde{\mmm}>\prod_{B\in\tilde{\pi}(F)} \mmm(T_B)$, meaning that \[\tilde{\mmm}=\text{in}_< \big(\tilde{\mmm}-\prod_{B\in\tilde{\pi}(F)} \mmm(T_B)\big)\in\text{in}_< I_{n,m}.\]
    It follows that $\tilde{\mmm}\notin\mathcal{SB}$. Therefore, our claim above is proved.

    Nonetheless, $\lvert\mathcal{SB}\rvert=\dim_\CC(R(\Pi_{n,m}))=\lvert\Pi_{n,m}\rvert$, which forces the containment claimed above to be an equality, finishing our proof.
\end{proof}

As the monomial bases in Corollary~\ref{cor:basis-presentation} modulo $I_{n,m}$ are all stable under the $\symm_n$-action, we obtain the graded $\symm_n$-module structure of $R(\Pi_{n,m})$.
\begin{proposition}\label{prop:general-module-str}
    For $d\ge 0$, we have that
    \[\Frob(R(\Pi_{n,m})_d)=\sum_{\substack{\lambda\vdash n\\ \lambda_1\le m \\ \ell(\lambda)=n-d}}SF_\lambda\]
    where $SF_\lambda \coloneqq \prod_{i=1}^r h_{a_i}[h_{b_i}]$ if we write $\lambda=(b_1^{a_1},b_2^{a_2},\dots,b_r^{a_r})$ with $b_1>b_2>\dots>b_r>0$. 
\end{proposition}
\begin{proof}
    Choose a family of forests $\{F_\pi\}_{\pi\in\Pi_{n,m}}$ on $[n]$ with $\tilde{\pi}(F_\pi)=\pi$ for all $\pi\in\Pi_{n,m}$. Corollary~\ref{cor:basis-presentation} states that $\{\mmm(F_\pi)\,:\,\pi\in\Pi_{n,m}\}$ descends to a basis of $R(\Pi_{n,m})$. Furthermore, Lemma~\ref{lem:diff-two-forests} reveals that this basis is compatible with the $\symm_n$-action, namely
    \[w\cdot\mmm(F_\pi)\equiv\mmm(F_{w\cdot\pi})\mod{I_{n,m}}\]
    for all $w\in\symm_n$. Note that for $\pi\in\Pi_\lambda$ with $\lambda\vdash n$ we have $\deg(\mmm(F_\pi))=\lvert E(F_\pi)\rvert = n-\ell(\lambda)$. Therefore, we deduce $\symm_n$-module isomorphisms
    \[R(\Pi_{n,m})_d\cong\CC\Bigg[\bigsqcup_{\substack{\lambda\vdash n\\ \lambda_1\le m\\ \ell(\lambda)=n-d}}\Pi_\lambda\Bigg]=\bigoplus_{\substack{\lambda\vdash n\\ \lambda_1\le m\\ \ell(\lambda)=n-d}}\CC[\Pi_\lambda]\]
    for $d\ge 0$, from which it follows that
    \[\Frob(R(\Pi_{n,m})_d)=\sum_{\substack{\lambda\vdash n\\ \lambda_1\le m\\ \ell(\lambda)=n-d}}\Frob(\CC[\Pi_\lambda]).\]
    Write $\lambda=(b_1^{a_1},b_2^{a_2},\dots,b_r^{a_r})$ with $b_1>b_2>\dots>b_r>0$. Then we have
    \[\CC[\Pi_\lambda]\cong\CC[\Pi_{(b_1^{a_1})}]\circ\dots\circ\CC[\Pi_{(b_r^{a_r})}]\]
    and thus
    \[\Frob(\CC[\Pi_\lambda]) =\prod_{i=1}^r h_{a_r}[h_{b_r}] = SF_\lambda,\]
    concluding our proof.
\end{proof}

\begin{remark}\label{rmk:restriction}
    For $\lambda\vdash n$ with $\lambda_1\le m$, we have $\Pi_\lambda\subseteq\Pi_{n,m}$ and hence $\gr\II(\Pi_{n,m})\subseteq\gr\II(\Pi_\lambda)$. Consequently, we have $\gr\Frob(R(\Pi_\lambda);q)\le\grFrob(R(\Pi_{n,m});q)$. As a result, the graded module structure in Proposition~\ref{prop:general-module-str} gives us an upper bound for $\grFrob(R(\Pi_{\lambda});q)$. If we take $\lambda=(b^a)$, this upper bound may help us understand $\grFrob(R(\Pi_{(b^a)});q)$ and hence understand the plethysm $h_a[h_b]=\grFrob(R(\Pi_{(b^a)});1)$.
\end{remark}

\section{Conclusion}\label{sec:conclusion}
In addition to Conjecture~\ref{conj:contain} tightly related to Foulkes' conjecture, we present some other future directions regarding log-concavity, Gröbner Theory and graded $\symm_n$-module structure.

To study Conjecture~\ref{conj:contain}, it may be better to first analyze the structure of $R(\Pi_{(b^a)})$.
\begin{problem}\label{prob:h_a[h_b]}
    For $a,b\ge 3$, find an explicit generating set of $\gr\II(\Pi_{(b^a)})$, a graded character formula of $R(\Pi_{(b^a)})$, and the standard monomial basis of $R(\Pi_{(b^a)})$ with respect to a strategic monomial order.
\end{problem}

A sequence $(a_1,\dots,a_m)$ of nonnegative numbers is called \emph{log-concave} if we have $a_i^2\ge a_{i-1}\cdot a_{i+1}$ for $1<i<m$. One way to generalize log-concavity is to incorporate equivariance with respect to a group $G$. Specifically, a sequence $(V_1,\dots,V_m)$ of $G$-modules is called \emph{$G$-log-concave} if we have $G$-module injections $V_{i-1}\otimes V_{i+1}\hookrightarrow V_i\otimes V_i$ for $1<i<m$, where $G$ diagonally acts on tensor products $V\otimes W$ of $G$-modules. We further call a graded $G$-module $V=\bigoplus_{d=0}^m V_m$ \emph{$G$-log-concave} if $(V_0,\dots,V_m)$ is $G$-log-concave. We raise an open problem about $\symm_n$-log-concavity.

For the graded $\symm_n$-modules studied in Subsections~\ref{subsec:b=2} and \ref{subsec:length-generalize}, we have verified the conjecture below for $n\le 20$ by coding.
\begin{conjecture}\label{conj:log-concavity}
    For integers $n>0$, both $\CC\Big[\xxxgraph{n}\Big]\Big/I(n)$ and $\CC\Big[\xxxgraph{n}\Big]\Big/J(n)$ are $\symm_n$-log-concave.
\end{conjecture}

For the graded $\symm_n$-modules $R(\Pi_{n,m})$ in Subsection~\ref{subsec:general}, the $\symm_n$-log-concavity fails when we have $n=7,8$ and $4\le m\le n$. However, according to our coding output, the $\symm_n$-log-concavity holds for $9\le n\le16$, so we have the following conjecture.
\begin{conjecture}\label{conj:general-log-concavity}
    For $n\ge 9$, $R(\Pi_{n,m})$ is $\symm_n$-log-concave.
\end{conjecture}

Another direction is to find standard monomial bases. While we have shown that $R(\Pi_{n,m})$ admits concise standard monomial bases with respect to any monomial order, we cannot do similar things to any rings occurring in Subsections~\ref{subsec:b=2} and \ref{subsec:length-generalize}.
\begin{problem}\label{prob:standard}
    Choose a strategic monomial order on $\CC\Big[\xxxgraph{n}\Big]$ and find the standard monomial basis of $\CC\Big[\xxxgraph{n}\Big]\Big/ I(n)$ (resp. $\CC\Big[\xxxgraph{n}\Big]\Big/J(n)$) with respect to this monomial order.
\end{problem}

Recall that $\CC\Big[\xxxgraph{n}\Big]\Big/I(n)=R(\Pi_{(2^{n/2})})$ and $\CC\Big[\xxxgraph{n}\Big]\Big/J(n)=R(\Pi_{(n/2)^2})$ for even integers $n$. If $n$ is odd, we ask for similar orbit harmonics constructions.
\begin{problem}\label{prob:orbit-harmonics}
    For odd integers $n>0$, find finite loci $\ZZZ_n,\ZZZ_n^\prime\subseteq\CC^{\binom{[n]}{2}}$ such that
    \[\CC\Big[\xxxgraph{n}\Big]\Big/I(n)=R(\ZZZ_n),\quad\CC\Big[\xxxgraph{n}\Big]\Big/J(n)=R(\ZZZ_n^\prime).\]
\end{problem}

\section{Acknowledgements}\label{sec:acknowledgements}
The author is thankful to Brendon Rhoades for inspiring conversations and suggestions.

\printbibliography

\end{document}